\documentclass[11pt]{article}

\usepackage[margin=1in]{geometry}
\usepackage[OT1]{fontenc}
\usepackage[latin1]{inputenc}
\usepackage[english]{babel}
\usepackage{amsmath,amssymb,amsfonts,amsthm,amsbsy,mathtools,bm}
\usepackage{amscd}
\usepackage{graphicx}
\usepackage{float}
\usepackage{subfig}
\usepackage{algorithm}
\usepackage{algpseudocode}
\usepackage[shortlabels]{enumitem}
\usepackage{indentfirst}
\usepackage{verbatim}
\usepackage{color}
\usepackage{hyperref}
\usepackage{matlab-prettifier}

\numberwithin{equation}{section}

\newtheorem{Theorem}{Theorem}
\newtheorem{Lemma}{Lemma}

\newtheorem{Proposition}{Proposition}

\newtheorem{Remark}{Remark}

\newtheorem*{Assumption*}{Assumption}

\newtheorem{Problem}{Problem}
\newtheorem*{Problem*}{Problem}

\begin{document}

\title{
Inverse initial data reconstruction for a memory convection--diffusion equation via Legendre spatial reduction and Tikhonov regularization
}

\author{
Cong B. Van
\thanks{Department of Mathematics and Statistics, University of North Carolina at
Charlotte, Charlotte, NC 28223, USA, \texttt{cvan1@charlotte.edu}, corresponding author}
\and
Thien P.B. Nguyen
\thanks{Department of Mathematics, International University - Vietnam National University, Quarter 33, Linh Xuan Ward, Ho Chi Minh City, Vietnam, \texttt{bthien624@gmail.com}}
\and
Minh-Binh Tran
\thanks{Department of Mathematics, Texas A\&M University, College Station, TX 77843, USA, \texttt{minhbinh@tamu.edu}}
\and
Loc H. Nguyen
\thanks{Department of Mathematics and Statistics, University of North Carolina at
Charlotte, Charlotte, NC 28223, USA, \texttt{loc.nguyen@charlotte.edu}}
}

\date{}

\maketitle


\begin{abstract}
We study an inverse initial data problem for a convection--diffusion equation with memory, where the goal is to recover the unknown initial condition from final-time data. The model includes convection, an instantaneous Laplacian term, and a nonlocal-in-time memory term involving the Laplacian of the past states, which leads to a severely ill-posed backward problem. We prove uniqueness in a spatially independent coefficient setting by applying the Fourier transform and using an analyticity argument for a scalar Volterra equation. For the variable-coefficient case, we develop a computational method based on Legendre spatial dimensional reduction and Tikhonov regularization. The solution is approximated by a finite tensor-product Legendre expansion, thereby reducing the inverse problem to a finite-dimensional terminal-value system for the time-dependent coefficients. We solve the reduced problem by a Tikhonov-regularized least-squares method with an $H^2$ penalty. For a fixed truncation order, we prove that the regularized minimizers converge to the finite-dimensional minimum-norm solution as the noise level and the regularization parameter vanish, under a suitable choice of the regularization parameter.  
Some two-dimensional numerical examples are presented to illustrate the performance of the proposed method.

\end{abstract}





\section{Introduction}

Let $d\geq 1$ be the spatial dimension and let $T>0$ be the final time. 
We consider the convection--diffusion equation with memory
\begin{equation}\label{eq:linear-convection-diffusion-memory}
    \begin{cases}
        u_t(\mathbf{x},t)
        + \mathbf{b}(\mathbf{x},t)\cdot \nabla u(\mathbf{x},t)
        =
        a(\mathbf{x},t)\Delta u(\mathbf{x},t)
        \\
        \hspace{3.5cm}
        +
        \displaystyle\int_0^t
        \alpha(\mathbf{x},t-\tau)\Delta u(\mathbf{x},\tau)\,d\tau,
        &
        (\mathbf{x},t)\in \mathbb{R}^d\times (0,T),
        \\[0.25cm]
        u(\mathbf{x},0)=u_0(\mathbf{x}),
        &
        \mathbf{x}\in \mathbb{R}^d.
    \end{cases}
\end{equation}
Here $u:\mathbb{R}^d\times[0,T]\to\mathbb{R}$ is the state variable, 
$\mathbf b:\mathbb{R}^d\times[0,T]\to\mathbb{R}^d$ is the convection field, 
$a:\mathbb{R}^d\times[0,T]\to [0, \infty)$ is the coefficient of the instantaneous Laplacian term, and 
$\alpha:\mathbb{R}^d\times[0,T]\to[0, \infty)$ is the memory kernel. 
The memory term represents the accumulated influence of previous diffusion states on the present evolution. 
Such terms arise naturally in models where the response is not instantaneous, for example in materials with relaxation effects, trapping mechanisms, heterogeneous microstructures, or nonlocal transport behavior.

We study the following inverse initial data problem.

\begin{Problem}[Backward inverse initial data problem]\label{prob:backward-inverse-initial}
Let $u$ solve \eqref{eq:linear-convection-diffusion-memory}. 
Given the final-time data
\[
    u_T(\mathbf{x})=u(\mathbf{x},T),
    \qquad \mathbf{x}\in\mathbb{R}^d,
\]
reconstruct the unknown initial profile $u_0(\mathbf{x})$.
\end{Problem}

Backward parabolic problems and final-time inverse problems are classical examples of ill-posed problems in the sense of Hadamard. 
Small perturbations in the final-time data may be strongly amplified when propagated backward in time, especially in high-frequency modes. 
This can already be seen in the simplest case $a=1$, $\alpha=0$, $\mathbf b=0$, and $d=1$, where \eqref{eq:linear-convection-diffusion-memory} reduces to the heat equation
\[
    u_t=u_{xx},
    \qquad x\in\mathbb{R},\quad t\in(0,T).
\]
Let $u_T^*$ be the exact final-time data and perturb it by
\[
    u_T^{(n)}(x)
    =
    u_T^*(x)
    +
    n e^{-n^2T}\sin(nx).
\]
Although the perturbation amplitude $n e^{-n^2T}$ tends to zero as $n\to\infty$, the corresponding reconstructed initial condition is
\[
    u_0^{(n)}(x)
    =
    u_0^*(x)+n\sin(nx).
\]
Thus an arbitrarily small perturbation in the final-time data may produce a large error in the reconstructed initial condition. 
This instability has been studied extensively in the classical literature on improperly posed problems, quasi-reversibility, and final value problems for evolution equations; see, for example,
\cite{EwingSIAMMA1975,LattesLionsBook1967,LavrentievRomanovShishatskiiAMS1986,MillerLNM1973,PayneSIAM1975,ShowalterJMAA1974}.
General references on inverse problems for partial differential equations and inverse problems in mathematical physics include
\cite{IsakovSpringer2006,PrilepkoOrlovskyVasinBook2000}.
Because of this instability, direct backward time integration is not suitable, and some form of regularization is necessary for stable reconstruction.

A large body of work has been devoted to the regularized reconstruction of initial data for parabolic equations. 
Classical approaches include quasi-reversibility, quasi-boundary value methods, Tikhonov regularization, spectral cut-off, kernel-based regularization, and optimization-based methods; see, for example,
\cite{AmesClarkEppersonOppenheimerMathComp1998,AmesEppersonSINUM1997,CarassoSIAMJAM1982,FuJCAM2004,HansenOLearySISC1993,HaoDucLesnicIMA2010,QuanTrietTrongIJMMS2012,TikhonovArseninBook1977,TrongTuanEJDE2006,QuanTrongTrietTuanIPSE2011,ChengFuActaMathSin2010,HouYangLiTianNumerAlgorithms2026}. 
Several numerical methods have also been developed for reconstructing the initial temperature or initial state from final-time measurements, including boundary element methods, methods of fundamental solutions, finite difference methods, finite element methods, conjugate-gradient methods, and space-time methods; see
\cite{ChenLiuJCAM2006,HanInghamYuanJCP1995,HasanovMuellerANM2001,KolevaVulkovComputation2023,LangerSteinbachTroeltzschYang2023,MeraIPSE2005,SamarskiiVabishchevichVasilievMathModel1997,SuHuangVasilievLiKardashevskyJCAM2022}. 
Backward and inverse problems for degenerate or variable-coefficient parabolic equations are also relevant to the present setting, especially when the diffusion coefficient is allowed to vanish or vary in space and time; see
\cite{AtifiEssoufiKhouitiOpusculaMath2020,YangDengNMPDE2017}. 
More recently, Carleman-estimate-based regularization, quasi-reversibility, convexification, and globally convergent methods have been developed for inverse and ill-posed parabolic problems, especially for reversed-time, nonlinear, or quasilinear equations; see, for instance,
\cite{Klibanov:anm2015,KlibanovYagolaInverseProblems2019,LeCON2023,LeNguyen:jiip2022,LiNguyenIPSE2019}.

Inverse problems for parabolic equations with memory have also attracted considerable attention. 
In such models, the present state depends not only on the instantaneous diffusion but also on the past history of the solution through a convolution-type memory kernel. 
Related works include inverse problems for determining memory kernels in parabolic integro-differential equations, heat equations with memory, and nonlocal transport or diffusion models; see
\cite{DurdievDifferentialEquations2014,DurdievZhumaevEJMCA2023,DurdievZhumaevMMAS2022,DurdievZhumaevUMJ2022,EfendievLeungLiPunVabishchevichJCP2023, WuGaoYanWuWangJDCS2018,JannoLorenziJIIP2008}. 
Backward problems for parabolic equations with memory have been studied by spectral regularization and related techniques; see
\cite{LuanKhanhApplicableAnalysis2021}. 
Fractional parabolic and diffusion-wave equations are also related to this direction, since fractional time derivatives can be interpreted as memory effects; see
\cite{HaoDucThangThanhJIIP2020,SuHuangVasilievLiKardashevskyJCAM2022,ZhangZhouSISC2022}.

The inverse problem considered in this paper has additional difficulties beyond the classical backward heat equation. 
First, the memory term makes the inverse problem more delicate. 
The evolution at time $t$ depends on the whole past history through
\[
    \int_0^t
    \alpha(\mathbf{x},t-\tau)\Delta u(\mathbf{x},\tau)\,d\tau.
\]
Thus, the dynamics are not governed solely by a local-in-time parabolic operator. 
In particular, when the coefficients depend on both space and time, a complete uniqueness and stability theory for recovering $u_0$ from $u_T$ is difficult and, to the best of our knowledge, is not available in the generality considered in the computational part of this paper.

Second, our computational framework does not require the coefficient $a(\mathbf{x},t)$ to be uniformly positive. 
In the standard uniformly parabolic setting, one assumes, for example,
\[
    a(\mathbf{x},t)\geq a_0>0.
\]
Here, in the computational part, $a$ is treated as a general known variable coefficient and may vanish in parts of the computational domain. 
Accordingly, the reduced finite-dimensional system should be viewed as a computational reconstruction model for a general linear integro-differential equation, rather than as a uniformly parabolic forward model.

Dimensional reduction has also been used as an effective tool in inverse and ill-posed problems. 
The main idea is to project the unknown function or the governing equation onto a finite-dimensional space, thereby reducing a high-dimensional problem to a more manageable system while retaining the dominant features of the solution. 
A related direction was developed in \cite{TrongElastic}, where a Legendre-polynomial exponential basis was introduced for inverse initial data problems. 
This type of basis has been successfully applied to several inverse problems for partial differential equations; see, for example,
\cite{LeVanDangNguyen, le2026globally,  neupane2026inverse, NguyenNguyenNguyen, nguyen2026carleman, VanLeNguyen}. 

Motivated by these works, we use a tensor-product Legendre basis in the spatial variables on the bounded computational domain $\Omega_R$. 
The solution is approximated by a finite Legendre expansion, which transforms the inverse initial data problem into a finite-dimensional terminal-value system for the time-dependent expansion coefficients. 
Because the original equation contains a memory term, the reduced system is an integro-differential system rather than a purely local-in-time system of ordinary differential equations. 
We then solve the reduced problem by a Tikhonov-regularized least-squares method with an $H^2$ penalty. 
The Legendre truncation also acts as a spectral filtering mechanism: by keeping only finitely many spatial modes, the method removes high-frequency components of the noisy terminal data. 
This is consistent with the Hadamard instability mechanism described above, where the strongest instability is associated with high-frequency perturbations.

The contributions of this paper are as follows.
\begin{itemize}
    \item We formulate an inverse initial data problem for a convection--diffusion equation with a memory term involving the Laplacian of the past states.
    
    \item We prove a uniqueness result in a simplified spatially independent setting. 
    This provides theoretical support for the inverse problem while keeping the proof accessible. 
    The fully variable-coefficient case is substantially more delicate and is treated computationally in this paper.
    
    \item We develop a Legendre spatial dimensional reduction method that converts the original problem into a finite-dimensional system of integro-differential equations for the expansion coefficients.
    
    \item We formulate a Tikhonov reconstruction method for the reduced system and prove convergence to the minimum-norm solution of the finite-dimensional exact problem as the noise level and regularization parameter vanish.
    
    \item We present two-dimensional numerical examples showing that the method can recover discontinuous initial profiles from noisy final-time data, including examples with convection, memory effects, and coefficients $a$ that are not uniformly positive.
\end{itemize}

The paper is organized as follows. 
Section \ref{sec:uniqueness} is devoted to the theoretical uniqueness result. 
There, we consider a simplified setting in which the coefficients are independent of the spatial variable and prove that the final-time data uniquely determine the initial data. 
Section \ref{sec2} returns to the variable-coefficient case and introduces the Legendre spatial dimensional reduction method. 
This section derives the finite-dimensional terminal value system for the Fourier--Legendre coefficients. 
Section \ref{sec:tikhonov-convergence} presents the Tikhonov reconstruction method for the reduced system and proves convergence of the regularized minimizers to the finite-dimensional minimum-norm solution. 
Section \ref{sec:numerics} describes the numerical implementation and presents two-dimensional reconstruction results. 
Finally, Section 6 gives concluding remarks.

\section{Uniqueness of the inverse initial data problem}
\label{sec:uniqueness}

We prove a uniqueness result in a setting where the coefficients are independent of the spatial variable, namely
\begin{equation}\label{eq:spatially-independent-coefficients}
    a=a(t), \qquad 
    \mathbf{b}=\mathbf{b}(t), \qquad \alpha=\alpha(t),
\end{equation}
for $t\in(0,T)$. 
This assumption decouples the equation in Fourier space and leads to a scalar Volterra equation for each frequency. 
The result below provides a theoretical justification in a tractable case; the numerical method in later sections is formulated for variable coefficients.

\begin{Remark}
The restriction in the uniqueness theorem is mainly due to the spatially dependent coefficients and the second-order memory term. 
When the coefficients depend on $\mathbf{x}$, the Fourier modes are coupled, and the term involving $\Delta u$ under the memory integral has the same spatial order as the principal diffusion term. 
A complete uniqueness and stability theory for the fully variable-coefficient case is beyond the scope of this paper.
\end{Remark}

\begin{Theorem}[Uniqueness in the spatially independent case]\label{thm:uniqueness-time-dependent}
Assume that
\[ a\in L^\infty(0,T), \qquad \mathbf{b}\in L^\infty(0,T;\mathbb{R}^d), \qquad \alpha\in L^1(0,T). \]
Consider the memory convection--diffusion equation
\begin{equation}\label{eq:memory-time-dependent}
    u_t(\mathbf{x},t)
    +
    \mathbf{b}(t)\cdot \nabla u(\mathbf{x},t)
    =
    a(t)\Delta u(\mathbf{x},t)
    +
    \int_0^t \alpha(t-\tau)\Delta u(\mathbf{x},\tau)\,d\tau,
    \quad
    (\mathbf{x},t)\in \mathbb{R}^d\times(0,T).
\end{equation}
Let $u^{(1)}$ and $u^{(2)}$ be two solutions of \eqref{eq:memory-time-dependent} satisfying \[ u^{(j)} \in W^{1,1}(0,T;L^2(\mathbb{R}^d)) \cap L^1(0,T;H^2(\mathbb{R}^d)), \qquad j=1,2. \]
If
\begin{equation}\label{eq:same-terminal-data}
    u^{(1)}(\cdot,T)=u^{(2)}(\cdot,T)
    \quad \mbox{in } L^2(\mathbb{R}^d),
\end{equation}
then
\[
    u^{(1)}(\cdot,0)=u^{(2)}(\cdot,0)
    \quad \mbox{in } L^2(\mathbb{R}^d).
\]
Consequently, the terminal data $u(\cdot,T)$ uniquely determine the initial data $u(\cdot,0)$ in this setting.
\end{Theorem}

We first establish two auxiliary lemmas for the scalar Volterra equations that will arise after applying the Fourier transform in the proof of Theorem \ref{thm:uniqueness-time-dependent}. 
The first lemma gives the well-posedness of the scalar Volterra initial value problem for each fixed frequency $\boldsymbol{\xi}$. 
The second lemma establishes the analytic dependence of the associated Fourier multiplier on the frequency variable. 
These two facts will be used later to prove that the terminal data uniquely determine the initial data.

\begin{Lemma}\label{lem:volterra-multiplier}
Under the assumptions of Theorem \ref{thm:uniqueness-time-dependent}, for each 
$\boldsymbol{\xi}\in\mathbb{R}^d$ and each $q_{\boldsymbol{\xi}}^0\in\mathbb{C}$, the scalar Volterra initial value problem
\begin{equation}\label{eq:q-xi-general}
    q_{\boldsymbol{\xi}}'(t)
    =
    -
    \Big(
        a(t)|\boldsymbol{\xi}|^2
        +
        i\,\mathbf{b}(t)\cdot\boldsymbol{\xi}
    \Big)
    q_{\boldsymbol{\xi}}(t)
    -
    |\boldsymbol{\xi}|^2
    \int_0^t
    \alpha(t-\tau)q_{\boldsymbol{\xi}}(\tau)\,d\tau,
    \qquad 
    q_{\boldsymbol{\xi}}(0)=q_{\boldsymbol{\xi}}^0,
\end{equation}
is uniquely solvable on $[0,T]$.
\end{Lemma}

\begin{proof}
Fix $\boldsymbol{\xi}\in\mathbb{R}^d$. For simplicity, write
\[
    c_{\boldsymbol{\xi}}(t)
    =
    a(t)|\boldsymbol{\xi}|^2
    +
    i\,\mathbf{b}(t)\cdot\boldsymbol{\xi}.
\]
Then \eqref{eq:q-xi-general} becomes
\[
    q_{\boldsymbol{\xi}}'(t)
    =
    -c_{\boldsymbol{\xi}}(t)q_{\boldsymbol{\xi}}(t)
    -
    |\boldsymbol{\xi}|^2
    \int_0^t
    \alpha(t-\tau)q_{\boldsymbol{\xi}}(\tau)\,d\tau,
    \qquad 
    q_{\boldsymbol{\xi}}(0)=q_{\boldsymbol{\xi}}^0.
\]
Integrating this equation from $0$ to $t$, we obtain
\begin{equation}\label{eq:q-integral-first}
    q_{\boldsymbol{\xi}}(t)
    =
    q_{\boldsymbol{\xi}}^0
    -
    \int_0^t c_{\boldsymbol{\xi}}(s)q_{\boldsymbol{\xi}}(s)\,ds
    -
    |\boldsymbol{\xi}|^2
    \int_0^t
    \int_0^s
    \alpha(s-\tau)q_{\boldsymbol{\xi}}(\tau)\,d\tau\,ds.
\end{equation}
Changing the order of integration in the last term gives
\[
\begin{aligned}
\int_0^t\int_0^s \alpha(s-\tau)q_{\boldsymbol{\xi}}(\tau)\,d\tau\,ds
&=
\int_0^t\left(\int_\tau^t \alpha(s-\tau)\,ds\right)
q_{\boldsymbol{\xi}}(\tau)\,d\tau  \\
&=
\int_0^t A(t-\tau)q_{\boldsymbol{\xi}}(\tau)\,d\tau,
\end{aligned}
\]
where
\[
    A(t)=\int_0^t \alpha(r)\,dr.
\]
Renaming $\tau$ as $s$, \eqref{eq:q-integral-first} becomes
\[
    q_{\boldsymbol{\xi}}(t)
    =
    q_{\boldsymbol{\xi}}^0
    -
    \int_0^t c_{\boldsymbol{\xi}}(s)q_{\boldsymbol{\xi}}(s)\,ds
    -
    |\boldsymbol{\xi}|^2
    \int_0^t
    A(t-s)q_{\boldsymbol{\xi}}(s)\,ds.
\]
Set
\[
    K_{\boldsymbol{\xi}}(t,s)
    =
    -c_{\boldsymbol{\xi}}(s)
    -
    |\boldsymbol{\xi}|^2 A(t-s),
    \qquad 0\leq s\leq t\leq T.
\]
Then the integral equation can be written as
\begin{equation}\label{eq:q-volterra-kernel}
    q_{\boldsymbol{\xi}}(t)
    =
    q_{\boldsymbol{\xi}}^0
    +
    \int_0^t K_{\boldsymbol{\xi}}(t,s)q_{\boldsymbol{\xi}}(s)\,ds.
\end{equation}
Since $\alpha\in L^1(0,T)$, the function $A$ is bounded on $[0,T]$. 
Also, $c_{\boldsymbol{\xi}}\in L^\infty(0,T)$ for each fixed $\boldsymbol{\xi}$. 
Hence there exists a constant $M_{\boldsymbol{\xi}}>0$ such that
\[
    |K_{\boldsymbol{\xi}}(t,s)|\leq M_{\boldsymbol{\xi}},
    \qquad 0\leq s\leq t\leq T.
\]

We now define a Picard sequence by
\[
    q_{\boldsymbol{\xi}}^{(0)}(t)=q_{\boldsymbol{\xi}}^0
\]
and
\begin{equation}\label{eq:picard-q}
    q_{\boldsymbol{\xi}}^{(n+1)}(t)
    =
    q_{\boldsymbol{\xi}}^0
    +
    \int_0^t K_{\boldsymbol{\xi}}(t,s)
    q_{\boldsymbol{\xi}}^{(n)}(s)\,ds,
    \qquad n\geq 0.
\end{equation}
A standard induction gives
\[
    \left|q_{\boldsymbol{\xi}}^{(n+1)}(t)-q_{\boldsymbol{\xi}}^{(n)}(t)\right|
    \leq
    |q_{\boldsymbol{\xi}}^0|
    \frac{(M_{\boldsymbol{\xi}}t)^{n+1}}{(n+1)!},
    \qquad 0\leq t\leq T.
\]
Since $0\leq t\leq T$, we also have
\[
    \left|q_{\boldsymbol{\xi}}^{(n+1)}(t)-q_{\boldsymbol{\xi}}^{(n)}(t)\right|
    \leq
    |q_{\boldsymbol{\xi}}^0|
    \frac{(M_{\boldsymbol{\xi}}T)^{n+1}}{(n+1)!},
    \qquad 0\leq t\leq T.
\]
The series
\[
    \sum_{n=0}^{\infty}
    |q_{\boldsymbol{\xi}}^0|
    \frac{(M_{\boldsymbol{\xi}}T)^{n+1}}{(n+1)!}
\]
is convergent. Hence the series
\[
    \sum_{n=0}^{\infty}
    \left(
        q_{\boldsymbol{\xi}}^{(n+1)}(t)
        -
        q_{\boldsymbol{\xi}}^{(n)}(t)
    \right)
\]
converges uniformly and absolutely on $[0,T]$. Since
\[
    q_{\boldsymbol{\xi}}^{(n)}(t)
    =
    q_{\boldsymbol{\xi}}^{(0)}(t)
    +
    \sum_{k=0}^{n-1}
    \left(
        q_{\boldsymbol{\xi}}^{(k+1)}(t)
        -
        q_{\boldsymbol{\xi}}^{(k)}(t)
    \right),
\]
we conclude that $\{q_{\boldsymbol{\xi}}^{(n)}\}_{n\geq 0}$ converges uniformly on $[0,T]$ to a function $q_{\boldsymbol{\xi}}$. 
Since each $q_{\boldsymbol{\xi}}^{(n)}$ is continuous, the limit $q_{\boldsymbol{\xi}}$ is also continuous on $[0,T]$. 
Passing to the limit in \eqref{eq:picard-q} gives
\[
    q_{\boldsymbol{\xi}}(t)
    =
    q_{\boldsymbol{\xi}}^0
    +
    \int_0^t K_{\boldsymbol{\xi}}(t,s)q_{\boldsymbol{\xi}}(s)\,ds.
\]
Thus, a solution exists.

To prove uniqueness, suppose that $q_{\boldsymbol{\xi}}^{(1)}$ and 
$q_{\boldsymbol{\xi}}^{(2)}$ are two solutions of \eqref{eq:q-xi-general} with the same initial value. 
Set
\[
    r(t)=q_{\boldsymbol{\xi}}^{(1)}(t)-q_{\boldsymbol{\xi}}^{(2)}(t).
\]
Then
\[
    r(t)=\int_0^t K_{\boldsymbol{\xi}}(t,s)r(s)\,ds, \qquad r(0) = 0.
\]
Using $|K_{\boldsymbol{\xi}}(t,s)|\leq M_{\boldsymbol{\xi}}$, we obtain
\begin{equation}\label{2.8}
    |r(t)|
    \leq
    M_{\boldsymbol{\xi}}\int_0^t |r(s)|\,ds,
    \qquad 0\leq t\leq T.
\end{equation}
Since $r(0)=0$, \eqref{2.8} can be written as
\[
    |r(t)|
    \leq
    |r(0)|
    +
    M_{\boldsymbol{\xi}}\int_0^t |r(s)|\,ds,
    \qquad 0\leq t\leq T.
\]
By Gronwall's inequality, this implies
\[
    |r(t)| \leq |r(0)|\exp(M_{\boldsymbol{\xi}}t)=0,
    \qquad 0\leq t\leq T.
\]
Hence $r\equiv 0$ on $[0,T]$, and the solution to \eqref{eq:q-xi-general} is unique.
\end{proof}

\begin{Lemma}\label{lem:analytic-multiplier}
Under the assumptions of Theorem \ref{thm:uniqueness-time-dependent}, for each 
$\boldsymbol{\xi}\in\mathbb{R}^d$, let $q_{\boldsymbol{\xi}}$ be the solution of
\begin{equation}\label{eq:q-xi-unit}
    q_{\boldsymbol{\xi}}'(t)
    =
    -
    \Big(
        a(t)|\boldsymbol{\xi}|^2
        +
        i\,\mathbf{b}(t)\cdot\boldsymbol{\xi}
    \Big)
    q_{\boldsymbol{\xi}}(t)
    -
    |\boldsymbol{\xi}|^2
    \int_0^t
    \alpha(t-\tau)q_{\boldsymbol{\xi}}(\tau)\,d\tau,
    \qquad 
    q_{\boldsymbol{\xi}}(0)=1.
\end{equation}
Then, for each fixed $t\in[0,T]$, the map
\[
    \boldsymbol{\xi}\mapsto q_{\boldsymbol{\xi}}(t)
\]
is real analytic on $\mathbb{R}^d$. Moreover,
\[
    q_{\mathbf{0}}(t)=1
    \qquad \text{for all } t\in[0,T].
\]
Consequently, $\boldsymbol{\xi}\mapsto q_{\boldsymbol{\xi}}(T)$ is not identically zero.
\end{Lemma}

\begin{proof}
The unique solvability of \eqref{eq:q-xi-unit} follows from Lemma 
\ref{lem:volterra-multiplier} by taking $q_{\boldsymbol{\xi}}^0=1$.

It remains to show the analytic dependence on $\boldsymbol{\xi}$. 
Recall from the proof of Lemma \ref{lem:volterra-multiplier} that \eqref{eq:q-xi-unit} is equivalent to the Volterra integral equation
\[
    q_{\boldsymbol{\xi}}(t)
    =
    1+
    \int_0^t K_{\boldsymbol{\xi}}(t,s)q_{\boldsymbol{\xi}}(s)\,ds,
\]
where
\[
    K_{\boldsymbol{\xi}}(t,s)
    =
    -
    \Big(
        a(s)|\boldsymbol{\xi}|^2
        +
        i\,\mathbf{b}(s)\cdot\boldsymbol{\xi}
    \Big)
    -
    |\boldsymbol{\xi}|^2 A(t-s),
    \qquad
    A(t)=\int_0^t \alpha(r)\,dr.
\]
The dependence of $K_{\boldsymbol{\xi}}(t,s)$ on $\boldsymbol{\xi}$ is polynomial. 
Equivalently, extend $\boldsymbol{\xi}\in\mathbb{R}^d$ to $\mathbf{z}\in\mathbb{C}^d$ and replace $|\boldsymbol{\xi}|^2$ by
\[
    \mathbf{z}\cdot\mathbf{z}=\sum_{j=1}^d z_j^2.
\]
The corresponding kernel is holomorphic in $\mathbf{z}$.

The Picard iteration \eqref{eq:picard-q} used in the proof of Lemma \ref{lem:volterra-multiplier} therefore generates functions that are holomorphic in $\mathbf{z}$. 
Moreover, the same estimates used in Lemma \ref{lem:volterra-multiplier} show that the Picard sequence converges uniformly on compact subsets of $\mathbb{C}^d$, uniformly for $t\in[0,T]$. 
Hence, by the Weierstrass theorem, the limit is holomorphic in $\mathbf{z}$. 
Restricting to real $\mathbf{z}=\boldsymbol{\xi}$, we conclude that 
$\boldsymbol{\xi}\mapsto q_{\boldsymbol{\xi}}(t)$ is real analytic on 
$\mathbb{R}^d$ for each fixed $t\in[0,T]$.

Finally, when $\boldsymbol{\xi}=\mathbf{0}$, equation \eqref{eq:q-xi-unit} reduces to
\[
    q_{\mathbf{0}}'(t)=0,
    \qquad q_{\mathbf{0}}(0)=1.
\]
Thus $q_{\mathbf{0}}(t)=1$ for all $t\in[0,T]$. 
In particular, $q_{\mathbf{0}}(T)=1$, and hence $q_{\boldsymbol{\xi}}(T)$ is not identically zero as a function of $\boldsymbol{\xi}$.
\end{proof}

\begin{proof}[Proof of Theorem \ref{thm:uniqueness-time-dependent}]
Define $w=u^{(2)}-u^{(1)}$. Then $w$ satisfies
\begin{equation}\label{eq:w-memory-equation}
     w_t(\mathbf{x},t)
        + \mathbf{b}(t)\cdot \nabla w(\mathbf{x},t)
        =
        a(t)\Delta w(\mathbf{x},t)
        +
        \int_0^t \alpha(t-\tau)\Delta w(\mathbf{x},\tau)\,d\tau,
\end{equation}
for $(\mathbf{x},t)\in \mathbb{R}^d\times(0,T)$.

We take the Fourier transform with respect to $\mathbf x$, using the convention
\[
\widehat w(\xi,t)=\int_{\mathbb R^d} w(\mathbf x,t)e^{-i\mathbf x\cdot \xi}\,d\mathbf x.
\]
Since $w(\cdot,t)$ is only assumed to belong to $L^2(\mathbb R^d)$, this formula is understood in the standard $L^2$ sense, that is, by the unitary extension of the Fourier transform from $\mathcal S(\mathbb R^d)$ to $L^2(\mathbb R^d)$. 
Then
\[
    \mathcal{F}_{\mathbf{x}}[\nabla w]
    =
    i\boldsymbol{\xi}\widehat w,
    \qquad
    \mathcal{F}_{\mathbf{x}}[\Delta w]
    =
    -|\boldsymbol{\xi}|^2\widehat w.
\]
Applying the Fourier transform to \eqref{eq:w-memory-equation}, we obtain
\[
    \partial_t \widehat w(\boldsymbol{\xi},t)
    +
    i\,\mathbf{b}(t)\cdot\boldsymbol{\xi}\,
    \widehat w(\boldsymbol{\xi},t)
    =
    -a(t)|\boldsymbol{\xi}|^2\widehat w(\boldsymbol{\xi},t)
    -
    |\boldsymbol{\xi}|^2
    \int_0^t
    \alpha(t-\tau)\widehat w(\boldsymbol{\xi},\tau)\,d\tau.
\]
Moreover, by \eqref{eq:same-terminal-data}, we have
\[
    w(\cdot,T)=0
    \quad \mbox{in } L^2(\mathbb{R}^d).
\]
Therefore,
\begin{equation}\label{eq:w-hat-terminal-zero}
    \widehat w(\boldsymbol{\xi},T)=0
    \qquad \text{for a.e. } \boldsymbol{\xi}\in\mathbb{R}^d.
\end{equation}

Let $q_{\boldsymbol{\xi}}$ be the solution to \eqref{eq:q-xi-unit}. 
Fix $\boldsymbol{\xi}\in\mathbb{R}^d$ for which the Fourier-mode equation is valid, and define
\[
    z(t)=q_{\boldsymbol{\xi}}(t)\widehat w(\boldsymbol{\xi},0),
    \qquad t\in[0,T].
\]
Since \eqref{eq:q-xi-unit} is linear and homogeneous, $z$ satisfies the scalar Volterra equation \eqref{eq:q-xi-general} with initial condition
\[
    z(0)=\widehat w(\boldsymbol{\xi},0).
\]
On the other hand, $\widehat w(\boldsymbol{\xi},t)$ also satisfies the same scalar Volterra equation with the same initial condition. 
By Lemma \ref{lem:volterra-multiplier}, we obtain
\[
    \widehat w(\boldsymbol{\xi},t)
    =
    q_{\boldsymbol{\xi}}(t)\widehat w(\boldsymbol{\xi},0),
    \qquad t\in[0,T].
\]
Taking $t=T$ and using \eqref{eq:w-hat-terminal-zero}, we get
\[
    0
    =
    \widehat w(\boldsymbol{\xi},T)
    =
    q_{\boldsymbol{\xi}}(T)\widehat w(\boldsymbol{\xi},0).
\]
By Lemma \ref{lem:analytic-multiplier}, the function 
$\boldsymbol{\xi}\mapsto q_{\boldsymbol{\xi}}(T)$ is real analytic and not identically zero. 
Therefore its zero set has Lebesgue measure zero. Hence
\[
    q_{\boldsymbol{\xi}}(T)\neq 0
    \quad \text{for a.e. } \boldsymbol{\xi}\in\mathbb{R}^d.
\]
It follows that
\[
    \widehat w(\boldsymbol{\xi},0)=0
    \quad \text{for a.e. } \boldsymbol{\xi}\in\mathbb{R}^d.
\]
Therefore,
\[
    w(\cdot,0)=0
    \quad \mbox{in } L^2(\mathbb{R}^d).
\]
Equivalently,
\[
    u^{(1)}(\cdot,0)=u^{(2)}(\cdot,0)
    \quad \mbox{in } L^2(\mathbb{R}^d).
\]
Thus, the initial data are uniquely determined by the terminal data.
\end{proof}

\begin{Remark}
The spatially independent setting in Theorem \ref{thm:uniqueness-time-dependent} is more restrictive than the variable-coefficient setting used in the numerical method in Section \ref{sec2} below, but it is still a relevant case. Many works on parabolic equations with memory consider constant or time-dependent diffusion coefficients and memory kernels depending only on the time lag; see, for example, \cite{DurdievZhumaevMMAS2022, DurdievZhumaevUMJ2022, LuanKhanhApplicableAnalysis2021}. \end{Remark}

\begin{Remark}
The spatial independence of the coefficients is used only in the uniqueness proof, where it allows the Fourier modes to decouple into scalar Volterra equations. This assumption is not required in the Legendre reduction or in the Tikhonov reconstruction method developed below. The numerical method is formulated for variable coefficients $a=a(\mathbf x,t)$, $\mathbf b=\mathbf b(\mathbf x,t)$, and $\alpha=\alpha(\mathbf x,t-\tau)$.
\end{Remark}

\section{Legendre spatial dimensional reduction method}
\label{sec2}

We now return to the variable-coefficient case used in the computational method:
\[
    a=a(\mathbf{x},t), 
    \qquad 
    \mathbf{b}=\mathbf{b}(\mathbf{x},t),
    \qquad 
    \alpha=\alpha(\mathbf{x},t-\tau).
\]
The reconstruction is performed in a bounded computational window
\[
    \Omega_R=(-R,R)^d.
\]
The goal of this section is to expand the solution in a tensor-product Legendre basis on $\Omega_R$ and derive a finite-dimensional system for the Fourier--Legendre coefficients.

Let $\{P_n\}_{n \ge 0}$ denote the classical Legendre polynomials on
$(-1,1)$.
We introduce the rescaled Legendre polynomials on $(-R,R)$ by
$$
Q_n(x) \;=\; P_n\!\left(\frac{x}{R}\right),
\qquad x \in (-R,R),\quad n \ge 0.
$$
An equivalent representation via the Rodrigues formula is
$$
Q_n(x)
\;=\;
\frac{1}{2^n n!\, R^n}
\frac{d^n}{dx^n}\bigl(x^2 - R^2\bigr)^n,
\qquad x \in (-R,R).
$$
The family $\{Q_n\}_{n \ge 0}$ satisfies the weighted orthogonality
relation
$$
\int_{-R}^{R} Q_m(x)\,Q_n(x)\,dx
\;=\;
\frac{2R}{2n+1}\,\delta_{mn},
\qquad m,\,n \ge 0,
$$
where $\delta_{mn}$ is the Kronecker delta.
The corresponding \emph{normalized} Legendre functions are defined by
$$
\Psi_n(x)
\;=\;
\left(\frac{2n+1}{2R}\right)^{\!1/2} Q_n(x)
\;=\;
\left(\frac{2n+1}{2R}\right)^{\!1/2}
P_n\!\left(\frac{x}{R}\right),
\qquad n \ge 0,
$$
and $\{\Psi_n\}_{n \ge 0}$ constitutes an orthonormal basis of
$L^2(-R,R)$:
$$
\int_{-R}^{R} \Psi_m(x)\,\Psi_n(x)\,dx \;=\; \delta_{mn},
\qquad m,\,n \ge 0.
$$

For each multi-index $\mathbf{n} = (n_1,\dots,n_d) \in \mathbb{N}^d$,
define the $d$-dimensional basis function
$$
\Phi_{\mathbf{n}}(\mathbf{x})
\;=\;
\prod_{j=1}^{d} \Psi_{n_j}(x_j),
\qquad \mathbf{x} = (x_1,\dots,x_d) \in \Omega_R.
$$
The family $\{\Phi_{\mathbf{n}}\}_{\mathbf{n} \in \mathbb{N}^d}$ forms a
complete orthonormal basis of $L^2(\Omega_R)$ by the tensor-product
structure of the one-dimensional basis.

For each $\mathbf{n} \in \mathbb{N}^d$, the Fourier--Legendre coefficient
of $u(\cdot,t)$ is defined by
\begin{equation}\label{coef}
	u_{\mathbf{n}}(t)
	\;=\;
	\int_{\Omega_R} u(\mathbf{x},t)\,\Phi_{\mathbf{n}}(\mathbf{x})\,d\mathbf{x}.
\end{equation}

The following convergence result is a standard consequence of the approximation theory for tensor-product Legendre projections. 
It follows directly from classical spectral approximation estimates; see, for example, \cite[Theorem~2.4]{CanutoQuarteroni}. 
The convergence of the differentiated expansions follows from the same approximation result together with the boundedness of the operators $\nabla:H^\mu(\Omega_R)\to H^{\mu-1}(\Omega_R)^d$ and $\Delta:H^\mu(\Omega_R)\to H^{\mu-2}(\Omega_R)$.

\begin{Proposition}[Convergence of the Fourier--Legendre expansion]
\label{prop:legendre-convergence}
Let $s>0$ and let $u\in L^2(0,T;H^s(\Omega_R))$. 
For $\mathbf{N}=(N_1,\dots,N_d)\in\mathbb{N}^d$, define the truncated Fourier--Legendre expansion
\begin{equation}\label{truncation}
    u^{\mathbf{N}}(\mathbf{x},t)
    =
    \sum_{\mathbf{n}\preceq\mathbf{N}}
    u_{\mathbf{n}}(t)\Phi_{\mathbf{n}}(\mathbf{x}),
    \qquad
    u_{\mathbf{n}}(t)
    =
    \int_{\Omega_R}u(\mathbf{x},t)\Phi_{\mathbf{n}}(\mathbf{x})\,d\mathbf{x},
\end{equation}
where $\mathbf{n}\preceq\mathbf{N}$ means $0\le n_j\le N_j$ for all $j=1,\dots,d$. 
Then, for every $0\le \mu<s$,
\[
    \|u-u^{\mathbf{N}}\|_{L^2(0,T;H^\mu(\Omega_R))}
    \to 0
    \qquad
    \text{as } \min_{1\le j\le d}N_j\to\infty.
\]
Equivalently,
\begin{equation}\label{trunc}
    u(\mathbf{x},t)
    =
    \sum_{\mathbf{n}\in\mathbb{N}^d}
    u_{\mathbf{n}}(t)\Phi_{\mathbf{n}}(\mathbf{x})
    \quad
    \text{in } L^2(0,T;H^\mu(\Omega_R)),
    \qquad 0\le \mu<s.
\end{equation}
Moreover, if $s>1$, then
\begin{equation}\label{2.3}
    \nabla u
    =
    \sum_{\mathbf{n}\in\mathbb{N}^d}
    u_{\mathbf{n}}(t)\nabla\Phi_{\mathbf{n}}
    \quad
    \text{in } L^2(0,T;H^{\mu-1}(\Omega_R)^d),
    \qquad 1<\mu<s,
\end{equation}
and if $s>2$, then
\begin{equation}\label{2.4}
    \Delta u
    =
    \sum_{\mathbf{n}\in\mathbb{N}^d}
    u_{\mathbf{n}}(t)\Delta\Phi_{\mathbf{n}}
    \quad
    \text{in } L^2(0,T;H^{\mu-2}(\Omega_R)),
    \qquad 2<\mu<s.
\end{equation}
\end{Proposition}

Substituting \eqref{trunc}, \eqref{2.3}, and \eqref{2.4} into
\eqref{eq:linear-convection-diffusion-memory} gives
\begin{align}
\label{sub-memory-full}
  &\sum_{\mathbf{n}\in\mathbb{N}^d}
  u_{\mathbf{n}}'(t)\Phi_{\mathbf{n}}(\mathbf{x})
  +
  \mathbf{b}(\mathbf{x},t)\cdot
  \sum_{\mathbf{n}\in\mathbb{N}^d}
  u_{\mathbf{n}}(t)\nabla\Phi_{\mathbf{n}}(\mathbf{x})
  \notag\\
  &\qquad =
  a(\mathbf{x},t)
  \sum_{\mathbf{n}\in\mathbb{N}^d}
  u_{\mathbf{n}}(t)\Delta\Phi_{\mathbf{n}}(\mathbf{x})
  +
  \int_0^t
  \alpha(\mathbf{x},t-\tau)
  \sum_{\mathbf{n}\in\mathbb{N}^d}
  u_{\mathbf{n}}(\tau)\Delta\Phi_{\mathbf{n}}(\mathbf{x})\,d\tau .
\end{align}
For each $\mathbf{m} \in \mathbb{N}^d$, multiplying \eqref{sub-memory-full} by $\Phi_{\mathbf{m}}(\mathbf{x})$,
integrating over $\Omega_R$, and using the orthonormality of
$\{\Phi_{\mathbf n}\}_{\mathbf n\in\mathbb N^d}$, we obtain, for each
$\mathbf{m}\in\mathbb{N}^d$,
\begin{align}
\label{coef-system-memory}
  u_{\mathbf{m}}'(t)
  &=
  \sum_{\mathbf{n}\in\mathbb{N}^d}
  C_{\mathbf{m}\mathbf{n}}(t)u_{\mathbf{n}}(t)
  +
  \sum_{\mathbf{n}\in\mathbb{N}^d}
  \int_0^t
  E_{\mathbf{m}\mathbf{n}}(t-\tau)u_{\mathbf{n}}(\tau)\,d\tau ,
  \qquad t\in(0,T),
\end{align}
where
\begin{align}
  A_{\mathbf{m}\mathbf{n}}(t)
  &:=
  \int_{\Omega_R}
  a(\mathbf{x},t)
  \Delta\Phi_{\mathbf{n}}(\mathbf{x})
  \Phi_{\mathbf{m}}(\mathbf{x})\,d\mathbf{x}, \label{Adef}
  \\
  B_{\mathbf{m}\mathbf{n}}(t)
  &:=
  \int_{\Omega_R}
  \bigl(\mathbf{b}(\mathbf{x},t)\cdot
  \nabla\Phi_{\mathbf{n}}(\mathbf{x})\bigr)
  \Phi_{\mathbf{m}}(\mathbf{x})\,d\mathbf{x},
  \\
  E_{\mathbf{m}\mathbf{n}}(s)
  &:=
  \int_{\Omega_R}
  \alpha(\mathbf{x},s)
  \Delta\Phi_{\mathbf{n}}(\mathbf{x})
  \Phi_{\mathbf{m}}(\mathbf{x})\,d\mathbf{x},
  \qquad s\in[0,T],
  \\
  C_{\mathbf{m}\mathbf{n}}(t)
  &:=
  A_{\mathbf{m}\mathbf{n}}(t)-B_{\mathbf{m}\mathbf{n}}(t).\label{Cdef}
\end{align}
Fix  a ``cutoff" vector $\mathbf{N} \in \mathbb{N}^d$. We write \eqref{coef-system-memory} as 
\begin{align}
\label{coef-system-memory1}
  {u_{\mathbf{m}}}'(t)
  &=
  \sum_{\mathbf{n}\preceq\mathbf{N}}
  C_{\mathbf{m}\mathbf{n}}(t)u_{\mathbf{n}}(t)
  +
  \sum_{\mathbf{n}\preceq\mathbf{N}}
  \int_0^t
  E_{\mathbf{m}\mathbf{n}}(t-\tau)u_{\mathbf{n}}(\tau)\,d\tau + R_{\mathbf{m}}^{\mathbf{N}}(t),
  \qquad t\in(0,T),
\end{align}
where
\[
	R_{\mathbf{m}}^{\mathbf{N}}(t)
  := \sum_{\mathbf{n} \npreceq \mathbf{N}}
     C_{\mathbf{m}\mathbf{n}}(t)\,u_{\mathbf{n}}(t)
     +
     \sum_{\mathbf{n} \npreceq \mathbf{N}}
     \int_0^t
     E_{\mathbf{m}\mathbf{n}}(t-\tau)u_{\mathbf{n}}(\tau)\,d\tau,
  \qquad \mathbf{m} \preceq \mathbf{N}.
\]
\begin{Proposition}[Decay of the truncation residual]
\label{lem:residual-decay}
Assume that $u \in L^2(0,T;\,H^s(\Omega_R))$ for some $s > 2$, and that
$a \in L^\infty(\Omega_R \times (0,T))$,
$\mathbf{b} \in L^\infty(\Omega_R \times (0,T))^d$, and
$\alpha \in L^\infty(\Omega_R\times(0,T))$.
For each truncation vector $\mathbf{N}$,  truncation residual  $R^{\mathbf{N}}
= \bigl(R_{\mathbf{m}}^{\mathbf{N}}\bigr)_{\mathbf{m} \preceq \mathbf{N}}
\in L^1(0,T;\,\mathbb{R}^{M_{\mathbf{N}}})$, and
$$
  \int_0^T
  \bigl|R^{\mathbf{N}}(t)\bigr|_{\mathbb{R}^{M_{\mathbf{N}}}}\,dt
  \longrightarrow 0
  \quad
  \text{as } \min_{1 \le j \le d} N_j \to \infty.
$$
\end{Proposition}
 
\begin{proof}

Let $P_N$ denote the $L^2(\Omega_R)$-orthogonal projection onto
\[
    \mathrm{span}\{\Phi_{\mathbf n}: \mathbf n\preceq N\}.
\]
Thus,
\[
    P_Nu(\mathbf x,t)=\sum_{\mathbf n\preceq N}u_{\mathbf n}(t)\Phi_{\mathbf n}(\mathbf x).
\]
Let
\[
    w^N(\mathbf x,t)=u(\mathbf x,t)-P_Nu(\mathbf x,t)
    =\sum_{\mathbf n\npreceq N}u_{\mathbf n}(t)\Phi_{\mathbf n}(\mathbf x)
\]
denote the projection error.  Using the definitions \eqref{Adef}--\eqref{Cdef},
$$
  R_{\mathbf{m}}^{\mathbf{N}}(t)
  = \int_{\Omega_R}
    \Bigl[
      a(\mathbf{x},t)\,\Delta w^{\mathbf{N}}(\mathbf{x},t)
      - \mathbf{b}(\mathbf{x},t)\cdot\nabla w^{\mathbf{N}}(\mathbf{x},t)
      + \int_0^t
        \alpha(\mathbf{x},t-\tau)
        \Delta w^{\mathbf{N}}(\mathbf{x},\tau)\,d\tau
    \Bigr]
    \Phi_{\mathbf{m}}(\mathbf{x})\,d\mathbf{x}.
$$
By Parseval's identity and the fact that orthogonal projection does
not increase the $L^2$-norm,
$$
  \bigl|R^{\mathbf{N}}(t)\bigr|_{\mathbb{R}^{M_{\mathbf{N}}}}
  \le
  \bigl\|
    a(\cdot,t)\,\Delta w^{\mathbf{N}}(\cdot,t)
    - \mathbf{b}(\cdot,t)\cdot\nabla w^{\mathbf{N}}(\cdot,t)
    + \int_0^t
      \alpha(\cdot,t-\tau)
      \Delta w^{\mathbf{N}}(\cdot,\tau)\,d\tau
  \bigr\|_{L^2(\Omega_R)}.
$$
Since $a$, $\mathbf{b}$, and $\alpha$ are essentially bounded, there exists
$\gamma > 0$, independent of $\mathbf{N}$, such that
$$
  \bigl|R^{\mathbf{N}}(t)\bigr|_{\mathbb{R}^{M_{\mathbf{N}}}}
  \le \gamma\,\|w^{\mathbf{N}}(\cdot,t)\|_{H^2(\Omega_R)}
  + \gamma\int_0^t
  \|w^{\mathbf{N}}(\cdot,\tau)\|_{H^2(\Omega_R)}\,d\tau,
  \quad \text{a.e. } t \in (0,T).
$$
Integrating in time and applying the Cauchy--Schwarz inequality,
$$
  \int_0^T
  \bigl|R^{\mathbf{N}}(t)\bigr|_{\mathbb{R}^{M_{\mathbf{N}}}}\,dt
  \le \gamma(\sqrt{T}+T^{3/2})\,
  \|u - P_{\mathbf{N}} u\|_{L^2(0,T;\,H^2(\Omega_R))}.
$$
Since $u \in L^2(0,T;\,H^s(\Omega_R))$ with $s > 2$,
Proposition~\ref{prop:legendre-convergence} gives
$P_{\mathbf{N}} u \to u$ in $L^2(0,T;\,H^2(\Omega_R))$,
from which the conclusion follows.
\end{proof}

\begin{Remark}
The assumption $u\in L^2(0,T;H^s(\Omega_R))$, $s>2$, in Proposition \ref{lem:residual-decay} is used only to justify the continuous Legendre truncation estimate. It is not a restriction of the finite-dimensional reconstruction algorithm. After projection, the method solves a Tikhonov-regularized least-squares problem and can be applied to nonsmooth data. The discontinuous examples in Section \ref{sec:numerics} are included to demonstrate this robustness beyond the smoothness class required for the theoretical residual estimate.
\end{Remark}

The terminal condition $u(\cdot,T)=u_T$ becomes
\begin{equation}
\label{terminal-full}
  u_{\mathbf{m}}(T) = u_{T, \mathbf{m}}
  :=
  \int_{\Omega_R}
  u_T(\mathbf{x})\Phi_{\mathbf{m}}(\mathbf{x})\,d\mathbf{x},
  \qquad \mathbf{m}\in\mathbb{N}^d.
\end{equation}
Thus, the backward inverse problem is transformed into an infinite-dimensional
terminal value problem with memory for the coefficient functions
$\{u_{\mathbf n}(t)\}_{\mathbf n\in\mathbb N^d}$.

Define
\[
    U^{\mathbf{N}}
    =
    \bigl(u_{\mathbf{m}}^{\mathbf{N}}\bigr)_{\mathbf{m}\preceq\mathbf{N}}.
\]
By \eqref{coef-system-memory1}, \eqref{terminal-full}, and Proposition
\ref{lem:residual-decay}, the vector $U^{\mathbf{N}}$ is approximately governed by the finite-dimensional terminal value system
\begin{equation}\label{reduced}
    \begin{cases}
    {u_{\mathbf{m}}}'(t)
    =
    \displaystyle\sum_{\mathbf{n}\preceq\mathbf{N}}
    C_{\mathbf{m}\mathbf{n}}(t)u_{\mathbf{n}}(t)
    +
    \displaystyle\sum_{\mathbf{n}\preceq\mathbf{N}}
    \int_0^t
    E_{\mathbf{m}\mathbf{n}}(t-\tau)u_{\mathbf{n}}(\tau)\,d\tau,
    & t\in(0,T),\\[0.2cm]
    u_{\mathbf{m}}(T)=u_{T,\mathbf{m}},
    &
    \mathbf{m}\preceq\mathbf{N}.
    \end{cases}
\end{equation}
We solve \eqref{reduced} by a Tikhonov-regularized least-squares method in the next section.

\section{The convergence of the Tikhonov method}
\label{sec:tikhonov-convergence}

We now study the convergence of the finite-dimensional Tikhonov reconstruction method. 
Throughout this section, the truncation vector $\mathbf{N}\in\mathbb{N}^d$ is fixed. 
For
\[
    V=(v_{\mathbf{m}})_{\mathbf{m}\preceq\mathbf{N}}
    \in H^2(0,T;\mathbb{R}^{M_{\mathbf{N}}}),
\]
define the linear residual operator
\[
    \mathcal L_{\mathbf{N}}V
    =
    \bigl((\mathcal L_{\mathbf{N}}V)_{\mathbf{m}}\bigr)_{\mathbf{m}\preceq\mathbf{N}},
\]
where
\[
    (\mathcal L_{\mathbf{N}}V)_{\mathbf{m}}(t)
    =
    v_{\mathbf{m}}'(t)
    -
    \sum_{\mathbf{n}\preceq\mathbf{N}}
    C_{\mathbf{m}\mathbf{n}}(t)v_{\mathbf{n}}(t)
    -
    \sum_{\mathbf{n}\preceq\mathbf{N}}
    \int_0^t
    E_{\mathbf{m}\mathbf{n}}(t-\tau)v_{\mathbf{n}}(\tau)\,d\tau,
    \qquad \mathbf{m}\preceq\mathbf{N}.
\]
Thus $\mathcal L_{\mathbf{N}}V$ measures the residual of the reduced integro-differential system. 
By the boundedness of the coefficients $a$, $\mathbf b$, and $\alpha$, and since the truncation vector $\mathbf N$ is fixed, the coefficient functions
$C_{\mathbf m\mathbf n}$ and $E_{\mathbf m\mathbf n}$ are bounded on $(0,T)$ for all $\mathbf m,\mathbf n\preceq \mathbf N$. 
Therefore, the operator
\[
    \mathcal L_{\mathbf{N}}:
    H^2(0,T;\mathbb{R}^{M_{\mathbf{N}}})
    \to
    L^2(0,T;\mathbb{R}^{M_{\mathbf{N}}})
\]
is bounded.

With this notation, the reduced terminal value system \eqref{reduced} can be written compactly as
\begin{equation}\label{compact-reduced}
    \mathcal L_{\mathbf{N}}U^{\mathbf{N}}=0,
    \qquad
    U^{\mathbf{N}}(T)=U_T^{\mathbf{N}},
\end{equation}
where
\[
    U_T^{\mathbf{N}}
    =
    (u_{T,\mathbf{m}})_{\mathbf{m}\preceq\mathbf{N}},
    \qquad
    u_{T,\mathbf{m}}
    =
    \int_{\Omega_R}u_T(\mathbf{x})\Phi_{\mathbf{m}}(\mathbf{x})\,d\mathbf{x}.
\]

In practice, the exact terminal data $u_T$ are not available. 
Instead, we are given noisy terminal data $u_T^\delta$. 
For each $\mathbf{m}\preceq\mathbf{N}$, define
\[
    u_{T,\delta,\mathbf{m}}
    =
    \int_{\Omega_R}u_T^\delta(\mathbf{x})\Phi_{\mathbf{m}}(\mathbf{x})\,d\mathbf{x},
\]
and set
\[
    U_{T,\delta}^{\mathbf{N}}
    =
    (u_{T,\delta,\mathbf{m}})_{\mathbf{m}\preceq\mathbf{N}}.
\]
We assume that the noise level satisfies
\begin{equation}\label{noise-terminal}
    \|U_{T,\delta}^{\mathbf{N}}-U_T^{\mathbf{N}}\|_{\mathbb{R}^{M_{\mathbf{N}}}}
    \leq \delta.
\end{equation}

Since the terminal data are prescribed, we impose the terminal condition as an exact constraint. 
For each $\delta>0$, define the affine space
\[
    \mathcal H_{\mathbf{N},\delta}
    =
    \left\{
    U\in H^2(0,T;\mathbb{R}^{M_{\mathbf{N}}})
    :
    U(T)=U_{T,\delta}^{\mathbf{N}}
    \right\}.
\]
For $\epsilon>0$, define the Tikhonov functional
\begin{equation}\label{tikhonov-functional}
    J_{\mathbf{N},\delta,\epsilon}(U)
    =
    \|\mathcal L_{\mathbf{N}}U\|_{L^2(0,T;\mathbb{R}^{M_{\mathbf{N}}})}^2
    +
    \epsilon
    \|U\|_{H^2(0,T;\mathbb{R}^{M_{\mathbf{N}}})}^2,
    \qquad
    U\in\mathcal H_{\mathbf{N},\delta}.
\end{equation}
The terminal condition is therefore enforced exactly in the admissible space, while the differential equation is imposed in the least-squares sense.

Since $J_{\mathbf{N},\delta,\epsilon}$ is strictly convex, coercive, and weakly lower semicontinuous on the closed affine subspace $\mathcal H_{\mathbf{N},\delta}$, it admits a unique minimizer. 
We denote this minimizer by
\[
    U_{\mathbf{N},\delta,\epsilon}
    =
    \operatorname*{arg\,min}_{U\in\mathcal H_{\mathbf{N},\delta}}
    J_{\mathbf{N},\delta,\epsilon}(U).
\]

\begin{Theorem}[Convergence of the Tikhonov method]
\label{thm:tikhonov-convergence}
Let $\mathbf{N}\in\mathbb{N}^d$ be fixed. 
Define the exact solution set
\[
    \mathcal S_{\mathbf{N}}
    =
    \left\{
    U\in H^2(0,T;\mathbb{R}^{M_{\mathbf{N}}}):
    \mathcal L_{\mathbf{N}}U=0
    \text{ in } L^2(0,T;\mathbb{R}^{M_{\mathbf{N}}}),
    \quad
    U(T)=U_T^{\mathbf{N}}
    \right\}.
\]
Assume that $\mathcal S_{\mathbf{N}}\neq\emptyset$. 
Let $U_*^{\mathbf{N}}$ be the minimum-norm element of $\mathcal S_{\mathbf{N}}$ in
$H^2(0,T;\mathbb{R}^{M_{\mathbf{N}}})$, namely
\[
    U_*^{\mathbf{N}}
    =
    \operatorname*{arg\,min}_{U\in\mathcal S_{\mathbf{N}}}
    \|U\|_{H^2(0,T;\mathbb{R}^{M_{\mathbf{N}}})}.
\]
If
\[
    \delta\to 0,
    \qquad
    \epsilon\to 0,
    \qquad
    \frac{\delta^2}{\epsilon}\to 0,
\]
then
\[
    U_{\mathbf{N},\delta,\epsilon}
    \to
    U_*^{\mathbf{N}}
    \quad
    \text{strongly in }
    H^2(0,T;\mathbb{R}^{M_{\mathbf{N}}}).
\]
\end{Theorem}

\begin{proof}
Set
\[
    \eta_\delta
    =
    U_{T,\delta}^{\mathbf{N}}-U_T^{\mathbf{N}}.
\]
By \eqref{noise-terminal},
\[
    |\eta_\delta|_{\mathbb{R}^{M_{\mathbf{N}}}}\leq \delta.
\]
Define the constant-in-time lifting
\[
    P_\delta(t)=\eta_\delta,
    \qquad 0\leq t\leq T.
\]
Then
\[
    P_\delta(T)=\eta_\delta,
    \qquad
    \|P_\delta\|_{H^2(0,T;\mathbb{R}^{M_{\mathbf{N}}})}
    \leq C\delta.
\]
Since $\mathcal L_{\mathbf{N}}$ is bounded, we also have
\[
    \|\mathcal L_{\mathbf{N}}P_\delta\|_{L^2(0,T;\mathbb{R}^{M_{\mathbf{N}}})}
    \leq C\delta.
\]

Because $U_*^{\mathbf{N}}(T)=U_T^{\mathbf{N}}$, the function
\[
    U_*^{\mathbf{N}}+P_\delta
\]
belongs to $\mathcal H_{\mathbf{N},\delta}$. 
Therefore, by the minimizing property of $U_{\mathbf{N},\delta,\epsilon}$,
\[
    J_{\mathbf{N},\delta,\epsilon}(U_{\mathbf{N},\delta,\epsilon})
    \leq
    J_{\mathbf{N},\delta,\epsilon}(U_*^{\mathbf{N}}+P_\delta).
\]
Since $\mathcal L_{\mathbf{N}}U_*^{\mathbf{N}}=0$, we obtain
\[
\begin{aligned}
    J_{\mathbf{N},\delta,\epsilon}(U_*^{\mathbf{N}}+P_\delta)
    &=
    \|\mathcal L_{\mathbf{N}}P_\delta\|_{L^2(0,T;\mathbb{R}^{M_{\mathbf{N}}})}^2
    +
    \epsilon
    \|U_*^{\mathbf{N}}+P_\delta\|_{H^2(0,T;\mathbb{R}^{M_{\mathbf{N}}})}^2    \\
    &\leq
    C\delta^2
    +
    \epsilon
    \left(
        \|U_*^{\mathbf{N}}\|_{H^2(0,T;\mathbb{R}^{M_{\mathbf{N}}})}
        +
        C\delta
    \right)^2.
\end{aligned}
\]
Thus
\begin{equation}\label{ineq:terminal-bound}
\begin{aligned}
    &\|\mathcal L_{\mathbf{N}}U_{\mathbf{N},\delta,\epsilon}\|_{L^2(0,T;\mathbb{R}^{M_{\mathbf{N}}})}^2
    +
    \epsilon
    \|U_{\mathbf{N},\delta,\epsilon}\|_{H^2(0,T;\mathbb{R}^{M_{\mathbf{N}}})}^2    \\
    &\qquad\leq
    C\delta^2
    +
    \epsilon
    \left(
        \|U_*^{\mathbf{N}}\|_{H^2(0,T;\mathbb{R}^{M_{\mathbf{N}}})}
        +
        C\delta
    \right)^2 .
\end{aligned}
\end{equation}
In particular,
\[
    \|U_{\mathbf{N},\delta,\epsilon}\|_{H^2(0,T;\mathbb{R}^{M_{\mathbf{N}}})}^2
    \leq
    C\frac{\delta^2}{\epsilon}
    +
    \left(
        \|U_*^{\mathbf{N}}\|_{H^2(0,T;\mathbb{R}^{M_{\mathbf{N}}})}
        +
        C\delta
    \right)^2 .
\]
Since $\delta^2/\epsilon\to 0$, the family
$\{U_{\mathbf{N},\delta,\epsilon}\}$ is bounded in
$H^2(0,T;\mathbb{R}^{M_{\mathbf{N}}})$.

Let $(\delta_k,\epsilon_k)$ be any sequence such that
\[
    \delta_k\to 0,
    \qquad
    \epsilon_k\to 0,
    \qquad
    \frac{\delta_k^2}{\epsilon_k}\to 0.
\]
By weak compactness, there exist a subsequence, still denoted by
$U_{\mathbf{N},\delta_k,\epsilon_k}$, and a function
$Z\in H^2(0,T;\mathbb{R}^{M_{\mathbf{N}}})$ such that
\[
    U_{\mathbf{N},\delta_k,\epsilon_k}
    \rightharpoonup Z
    \quad
    \text{weakly in }
    H^2(0,T;\mathbb{R}^{M_{\mathbf{N}}}).
\]
Since $\mathcal L_{\mathbf{N}}$ is linear and bounded,
\[
    \mathcal L_{\mathbf{N}}U_{\mathbf{N},\delta_k,\epsilon_k}
    \rightharpoonup
    \mathcal L_{\mathbf{N}}Z
    \quad
    \text{weakly in }
    L^2(0,T;\mathbb{R}^{M_{\mathbf{N}}}).
\]
On the other hand, \eqref{ineq:terminal-bound} gives
\[
    \|\mathcal L_{\mathbf{N}}U_{\mathbf{N},\delta_k,\epsilon_k}\|_{L^2(0,T;\mathbb{R}^{M_{\mathbf{N}}})}^2
    \leq
    C\delta_k^2
    +
    \epsilon_k
    \left(
        \|U_*^{\mathbf{N}}\|_{H^2(0,T;\mathbb{R}^{M_{\mathbf{N}}})}
        +
        C\delta_k
    \right)^2
    \to 0.
\]
Hence
\[
    \mathcal L_{\mathbf{N}}Z=0
    \quad
    \text{in }
    L^2(0,T;\mathbb{R}^{M_{\mathbf{N}}}).
\]

Since $U_{\mathbf{N},\delta_k,\epsilon_k}\in\mathcal H_{\mathbf{N},\delta_k}$, we have
\[
    U_{\mathbf{N},\delta_k,\epsilon_k}(T)
    =
    U_{T,\delta_k}^{\mathbf{N}}.
\]
By \eqref{noise-terminal},
\[
    U_{T,\delta_k}^{\mathbf{N}}\to U_T^{\mathbf{N}}
    \quad
    \text{in }
    \mathbb{R}^{M_{\mathbf{N}}}.
\]
The trace map $U\mapsto U(T)$ is continuous from
$H^2(0,T;\mathbb{R}^{M_{\mathbf{N}}})$ to $\mathbb{R}^{M_{\mathbf{N}}}$; hence, passing to the weak limit gives
\[
    Z(T)=U_T^{\mathbf{N}}.
\]
Therefore
\[
    Z\in\mathcal S_{\mathbf{N}}.
\]

By weak lower semicontinuity and \eqref{ineq:terminal-bound},
\[
\begin{aligned}
    \|Z\|_{H^2(0,T;\mathbb{R}^{M_{\mathbf{N}}})}
    &\leq
    \liminf_{k\to\infty}
    \|U_{\mathbf{N},\delta_k,\epsilon_k}\|_{H^2(0,T;\mathbb{R}^{M_{\mathbf{N}}})}  \\
    &\leq
    \limsup_{k\to\infty}
    \|U_{\mathbf{N},\delta_k,\epsilon_k}\|_{H^2(0,T;\mathbb{R}^{M_{\mathbf{N}}})}  \\
    &\leq
    \|U_*^{\mathbf{N}}\|_{H^2(0,T;\mathbb{R}^{M_{\mathbf{N}}})}.
\end{aligned}
\]
Since $Z\in\mathcal S_{\mathbf{N}}$ and $U_*^{\mathbf{N}}$ is the minimum-norm element of $\mathcal S_{\mathbf{N}}$, we obtain
\[
    Z=U_*^{\mathbf{N}}.
\]
Thus every weakly convergent subsequence has the same weak limit $U_*^{\mathbf{N}}$, and therefore
\[
    U_{\mathbf{N},\delta,\epsilon}
    \rightharpoonup
    U_*^{\mathbf{N}}
    \quad
    \text{weakly in }
    H^2(0,T;\mathbb{R}^{M_{\mathbf{N}}}).
\]

Finally, \eqref{ineq:terminal-bound} gives
\[
    \limsup_{\delta,\epsilon\to 0}
    \|U_{\mathbf{N},\delta,\epsilon}\|_{H^2(0,T;\mathbb{R}^{M_{\mathbf{N}}})}
    \leq
    \|U_*^{\mathbf{N}}\|_{H^2(0,T;\mathbb{R}^{M_{\mathbf{N}}})}.
\]
Together with weak lower semicontinuity, this implies
\[
    \lim_{\delta,\epsilon\to 0}
    \|U_{\mathbf{N},\delta,\epsilon}\|_{H^2(0,T;\mathbb{R}^{M_{\mathbf{N}}})}
    =
    \|U_*^{\mathbf{N}}\|_{H^2(0,T;\mathbb{R}^{M_{\mathbf{N}}})}.
\]
Weak convergence together with convergence of norms in the Hilbert space
$H^2(0,T;\mathbb{R}^{M_{\mathbf{N}}})$ implies strong convergence. Therefore,
\[
    U_{\mathbf{N},\delta,\epsilon}
    \to
    U_*^{\mathbf{N}}
    \quad
    \text{strongly in }
    H^2(0,T;\mathbb{R}^{M_{\mathbf{N}}}).
\]
The proof is complete.
\end{proof}

\begin{Remark}
The fixed truncation order $N$ in Theorem \ref{thm:tikhonov-convergence} is not merely a technical assumption. It is part of the regularization mechanism. As illustrated by the instability example in the introduction, the severe ill-posedness of the backward problem is caused by the amplification of highly oscillatory components in the terminal data. Therefore, allowing $N\to\infty$ without additional regularization would reintroduce increasingly high spatial modes and would move the reduced problem toward the full backward problem, which cannot be solved stably from noisy terminal data. For this reason, the study of the limit $N\to\infty$ in the inverse reconstruction is outside the scope of the present paper. In the proposed method, $N$ is kept finite so that the Legendre truncation filters out the high-frequency components responsible for the instability.
\end{Remark}

\section{Numerical studies}
\label{sec:numerics}

In this section, we present numerical experiments to test the performance of the proposed Legendre spatial dimensional reduction method combined with Tikhonov regularization. 
The experiments are carried out in two spatial dimensions. 
The first step is to solve a forward problem in order to generate synthetic terminal data. 
The reconstructed initial condition is then computed from the noisy terminal data by the proposed inverse solver.

\subsection{Generation of synthetic data by a forward solver}

The synthetic data used in the numerical experiments are generated by solving the forward problem on a large square domain
\[
    \Omega_G=(-G,G)^2
\]
in order to approximate the whole-space model posed on $\mathbb{R}^2$. 
After the forward solution is computed, the terminal data are restricted to the smaller computational window
\[
    \Omega_R=(-R,R)^2,
    \qquad R<G,
\]
where the inverse reconstruction is performed.

In the numerical tests, we use
\[
    G=15,\qquad R=8,\qquad \Delta x=\Delta y=0.1,
\]
and
\[
    T=1,\qquad \Delta t=10^{-4}.
\]
Thus the forward grid is defined on $\Omega_G$, while the inverse grid is obtained by selecting the grid points satisfying
\[
    |x|\leq R,\qquad |y|\leq R.
\]

We generate the terminal data by solving the two-dimensional equation
\begin{multline}\label{eq:numerical-forward-model}
    u_t(x,y,t)
    +
    b_x(x,y,t)u_x(x,y,t)
    +
    b_y(x,y,t)u_y(x,y,t)
    \\
    =
    a(x,y,t)\Delta u(x,y,t)
    +
    \int_0^t
    \alpha(x,y,t-\tau)\Delta u(x,y,\tau)\,d\tau .
\end{multline}
Here $a$, $b_x$, $b_y$, and $\alpha$ are prescribed coefficient functions. 
The function $u(x,y,0)=u_0(x,y)$ is the exact initial condition. These functions will be specified later in each numerical test.

Let $u_{i,j}^k$ denote the numerical approximation of $u(x_i,y_j,t_k)$, where
\[
    t_k=k\Delta t.
\]
The Laplacian is approximated by the standard central difference formula
\[
    \Delta_h u_{i,j}^k
    =
    \frac{u_{i+1,j}^k-2u_{i,j}^k+u_{i-1,j}^k}{\Delta x^2}
    +
    \frac{u_{i,j+1}^k-2u_{i,j}^k+u_{i,j-1}^k}{\Delta y^2}.
\]
The convection terms are approximated by an upwind scheme. For example,
\[
    (D_x^{\rm up}u)_{i,j}^k
    =
    \begin{cases}
    \displaystyle
    \frac{u_{i,j}^k-u_{i-1,j}^k}{\Delta x},
    & b_x(x_i,y_j,t_k)\geq 0,\\[0.3cm]
    \displaystyle
    \frac{u_{i+1,j}^k-u_{i,j}^k}{\Delta x},
    & b_x(x_i,y_j,t_k)<0.
    \end{cases}
\]
The derivative $(D_y^{\rm up}u)_{i,j}^k$ is defined analogously using the sign of $b_y(x_i,y_j,t_k)$.

The memory term is evaluated explicitly by a left Riemann sum. At time $t_k$, we set
\[
    M_{i,j}^k
    =
    \Delta t
    \sum_{\ell=0}^{k-1}
    \alpha(x_i,y_j,t_k-t_\ell)\Delta_h u_{i,j}^{\ell}.
\]
Thus $M_{i,j}^0=0$, and only previously computed time levels are used. The explicit update is then
\[
\begin{aligned}
    u_{i,j}^{k+1}
    =
    u_{i,j}^{k}
    +
    \Delta t
    \big[
        &a(x_i,y_j,t_k)\Delta_h u_{i,j}^{k}
        +
        M_{i,j}^{k}  \\
        &-
        b_x(x_i,y_j,t_k)(D_x^{\rm up}u)_{i,j}^{k}
        -
        b_y(x_i,y_j,t_k)(D_y^{\rm up}u)_{i,j}^{k}
    \big].
\end{aligned}
\]
After each time step, the boundary values on $\partial\Omega_G$ are assigned by linear extrapolation from the interior grid values. For example,
\[
    u_{1,j}^{k+1}=2u_{2,j}^{k+1}-u_{3,j}^{k+1},
    \qquad
    u_{N_y,j}^{k+1}=2u_{N_y-1,j}^{k+1}-u_{N_y-2,j}^{k+1},
\]
and the vertical boundaries are treated in the same way.

After solving \eqref{eq:numerical-forward-model} up to time $T$, we define the exact terminal data on the inverse domain $\Omega_R$ by
\[
    u_T^*(x,y)=u(x,y,T),
    \qquad (x,y)\in\Omega_R.
\]

To simulate measurement noise, we perturb the terminal data by
\[
    u_T^\delta(x,y)
    =
    u_T^*(x,y)
    \left(
        1+\frac{\delta}{100}\eta(x,y)
    \right),
\]
where $\delta$ is the noise level in percent and $\eta(x,y)$ is uniformly distributed in $[-1,1]$. 
The noisy terminal data $u_T^\delta$ are then used as input for the inverse reconstruction.

\subsection{Implementation of the inverse solver}

We now describe the implementation of the inverse reconstruction on the domain
$
    \Omega_R=(-R,R)^2.
$
Let
$
    \mathbf N=(N,N)
$
be the truncation order. In the numerical tests below, we use
$
    N=15.
$
We first construct the tensor-product Legendre basis
$
    \{\Phi_{\mathbf m}\}_{\mathbf m\preceq \mathbf N}
$
together with its spatial derivatives
\[
    \partial_x\Phi_{\mathbf m},\qquad
    \partial_y\Phi_{\mathbf m},\qquad
    \Delta\Phi_{\mathbf m}.
\]
The noisy terminal data $u_T^\delta$ are projected onto this basis to obtain
\[
    u_{T,\delta,\mathbf m}
    =
    \int_{\Omega_R}
    u_T^\delta(x,y)\Phi_{\mathbf m}(x,y)\,dx\,dy,
    \qquad \mathbf m\preceq\mathbf N.
\]
We denote the corresponding coefficient vector by
\[
    U_{T,\delta}^{\mathbf N}
    =
    (u_{T,\delta,\mathbf m})_{\mathbf m\preceq\mathbf N}.
\]

The inverse problem is solved on the time grid
\[
    0=t_0<t_1<\cdots<t_K=T,
    \qquad t_k=k\Delta t_{\rm inv},
\]
with
\[
    \Delta t_{\rm inv}=10^{-2}.
\]
Let $U^k\in\mathbb R^{M_{\mathbf N}}$ be the vector of Legendre coefficients at time $t_k$. 
The terminal condition is imposed exactly by setting
\[
    U^K=U_{T,\delta}^{\mathbf N}.
\]
Thus, the unknowns are only
\[
    U^0,U^1,\dots,U^{K-1}.
\]

At each time level, we assemble the reduced matrices
$
    C^k=\bigl(C_{\mathbf m\mathbf n}(t_k)\bigr)_{\mathbf m,\mathbf n\preceq\mathbf N},
$
and the memory matrices
$
    E^\ell=\bigl(E_{\mathbf m\mathbf n}(t_\ell)\bigr)_{\mathbf m,\mathbf n\preceq\mathbf N}.
$
The discrete residual is
\[
    \mathcal R^k(U)
    =
    \frac{U^{k+1}-U^k}{\Delta t_{\rm inv}}
    -
    C^kU^k
    -
    \Delta t_{\rm inv}
    \sum_{j=0}^{k-1}
    E^{k-j}U^j,
    \qquad k=0,\dots,K-1.
\]
We minimize the discrete Tikhonov functional
\[
    J_h(U)
    =
    \Delta t_{\rm inv}
    \sum_{k=0}^{K-1}|\mathcal R^k(U)|^2
    +
    \epsilon \|U\|_{H_h^2}^2
\]
subject to the terminal constraint $U^K=U_{T,\delta}^{\mathbf N}$, where
\[
\begin{aligned}
    \|U\|_{H_h^2}^2
    =
    &\Delta t_{\rm inv}\sum_{k=0}^K |U^k|^2
    +
    \frac{1}{\Delta t_{\rm inv}}
    \sum_{k=0}^{K-1}|U^{k+1}-U^k|^2  \\
    &+
    \frac{1}{\Delta t_{\rm inv}^3}
    \sum_{k=0}^{K-2}
    |U^{k+2}-2U^{k+1}+U^k|^2 .
\end{aligned}
\]
In computation, we use
$
    \epsilon=10^{-5}.
$
The resulting linear least-squares problem is solved by MATLAB's \texttt{lsqr} routine with tolerance $10^{-8}$ and maximum number of iterations $300$.

\begin{Remark}[Choice of the truncation and regularization parameters]
In the numerical experiments, the truncation order $N$ and the regularization parameter $\epsilon$ were chosen empirically by trial and error. 
We used Test 1 as a reference test and selected values that produced a stable reconstruction with a good balance between noise suppression and preservation of the main shape of the target. 
After these values were selected, they were kept fixed for all subsequent tests. 
In particular, we used
\[
    N=15,
    \qquad
    \epsilon=10^{-5}
\]
throughout the numerical experiments. 
This choice avoids tuning the truncation and regularization parameters separately for each example and provides a consistent test of the proposed reconstruction method.
\end{Remark}

After computing the coefficient sequence $U_{\rm rec}^0,\dots,U_{\rm rec}^K$, we reconstruct
\[
    u_{\rm rec}^{\mathbf N}(x,y,t_k)
    =
    \sum_{\mathbf m\preceq\mathbf N}
    (U_{\rm rec}^k)_{\mathbf m}\Phi_{\mathbf m}(x,y).
\]
The reconstructed initial condition is
\[
    u_0^{\rm rec}(x,y)=u_{\rm rec}^{\mathbf N}(x,y,0).
\]

For clarity, we summarize the implementation of the inverse reconstruction in Algorithm \ref{alg:inverse-implementation}.
\begin{algorithm}[H]
\caption{Inverse reconstruction}
\label{alg:inverse-implementation}
\begin{algorithmic}[1]
\State Construct the Legendre basis $\{\Phi_{\mathbf m}\}_{\mathbf m\preceq\mathbf N}$ on $\Omega_R$.
\State Compute the terminal coefficients $U_{T,\delta}^{\mathbf N}$ from $u_T^\delta$.
\State Assemble the reduced matrices $C^k$ and $E^\ell$.
\State Impose the terminal constraint $U^K=U_{T,\delta}^{\mathbf N}$.
\State Minimize
\[
    J_h(U)
    =
    \Delta t_{\rm inv}\sum_{k=0}^{K-1}|\mathcal R^k(U)|^2
    +
    \epsilon\|U\|_{H_h^2}^2
\]
over $U^0,\dots,U^{K-1}$ by \texttt{lsqr}.
\State Reconstruct
\[
    u_{\rm rec}^{\mathbf N}(x,y,t_k)
    =
    \sum_{\mathbf m\preceq\mathbf N}
    (U_{\rm rec}^k)_{\mathbf m}\Phi_{\mathbf m}(x,y).
\]
\State Output $u_0^{\rm rec}(x,y)=u_{\rm rec}^{\mathbf N}(x,y,0)$.
\end{algorithmic}
\end{algorithm}

\subsection{Numerical examples}
\label{subsec:numerical-examples}

We now present several numerical examples to demonstrate the performance of the proposed reconstruction method. 
In each test, we prescribe the coefficient functions $a$, $b_x$, $b_y$, the memory kernel $\alpha$, and the exact initial condition $u_0$. 
The quality of the reconstruction is evaluated by comparing the reconstructed initial condition $u_0^{\rm rec}$ with the exact initial condition $u_0$ on $\Omega_R$.

\textbf{Test 1.} 
In the first test, the exact initial condition is a circular inclusion centered at $(2, -2)$. 
More precisely, we set
\[
    u_0(x,y)
    =
    \begin{cases}
    1, & (x-2)^2+(y+2)^2\leq 2^2,\\
    0, & \mbox{otherwise}.
    \end{cases}
\]
This test evaluates whether the proposed method can recover a simple, localized, discontinuous profile from noisy terminal data.

The coefficient functions are chosen as
\[
    a(x,y,t)
    =
    \left[
    1
    +
    0.2\sin\left(\frac{\pi x}{15}\right)
    \cos\left(\frac{\pi y}{15}\right)
    \right]
    \left[
    1+0.12\sin(2\pi t)
    \right].
\]
The convection field is given by
\[
    b_x(x,y,t)
    =
    0.5\cos\left(\frac{\pi x}{15}\right)
    \left[
    1+0.20\sin\left(2\pi t+\frac{\pi}{4}\right)
    \right],
\]
and
\[
    b_y(x,y,t)
    =
    0.5\sin\left(\frac{\pi y}{15}\right)
    \left[
    1+0.20\sin\left(2\pi t+\frac{\pi}{2}\right)
    \right].
\]
The memory kernel depends on the lag variable $s=t-\tau$ and is defined by
\[
    \alpha(x,y,s)
    =
    \left[
    0.15
    +
    0.05\cos\left(\frac{\pi x}{15}\right)
    \cos\left(\frac{\pi y}{15}\right)
    \right]
    e^{-s}
    \left[
    1+0.25\cos(2\pi s)
    \right],
    \qquad s\geq 0.
\]

\begin{figure}[H]
    \subfloat[Exact initial condition $u_0$]{
        \includegraphics[width=0.31\textwidth]{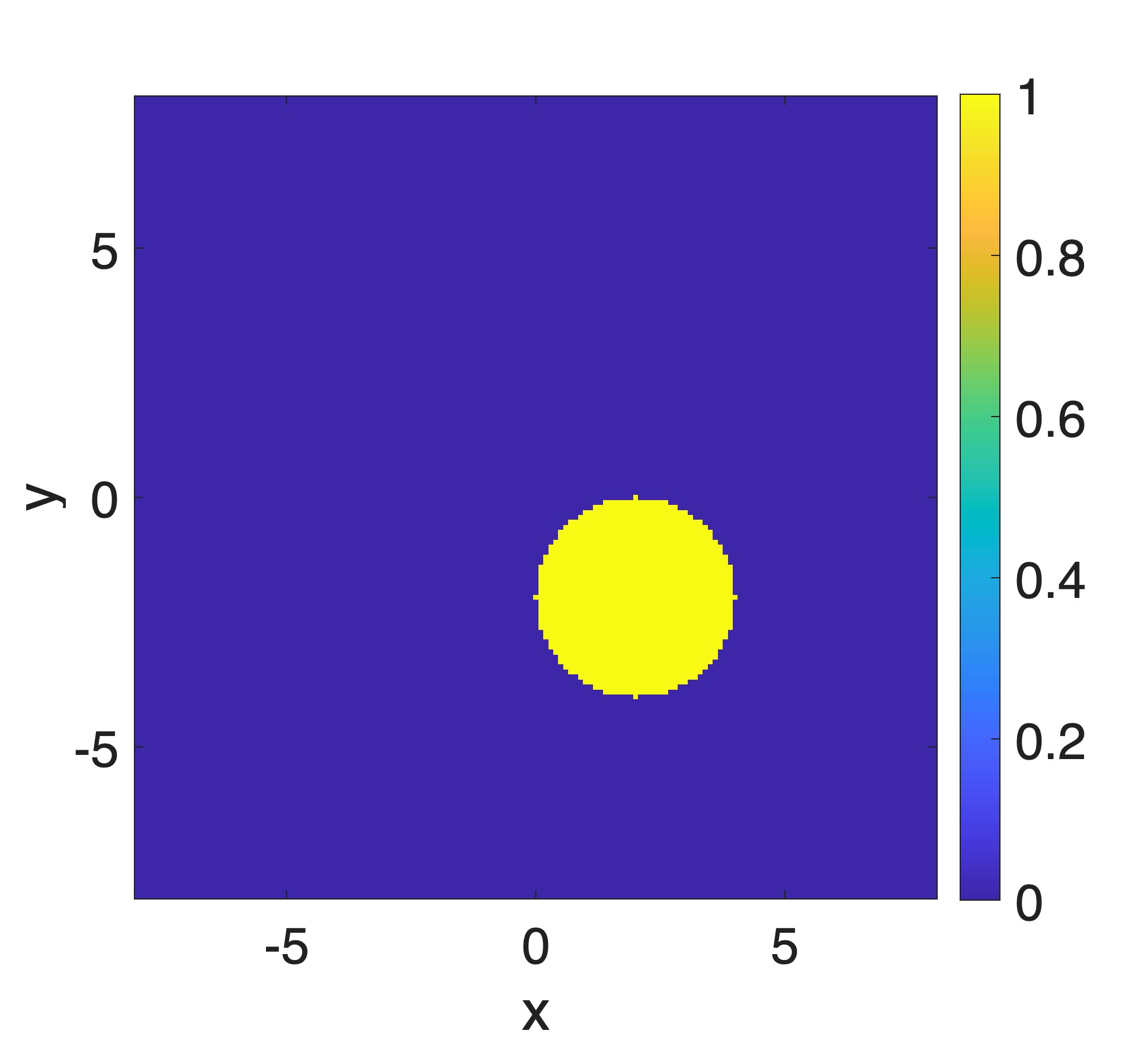}
    }
    \hfill
    \subfloat[Reconstructed initial condition $u_0^{\rm rec}$, $10\%$ noise]{
        \includegraphics[width=0.31\textwidth]{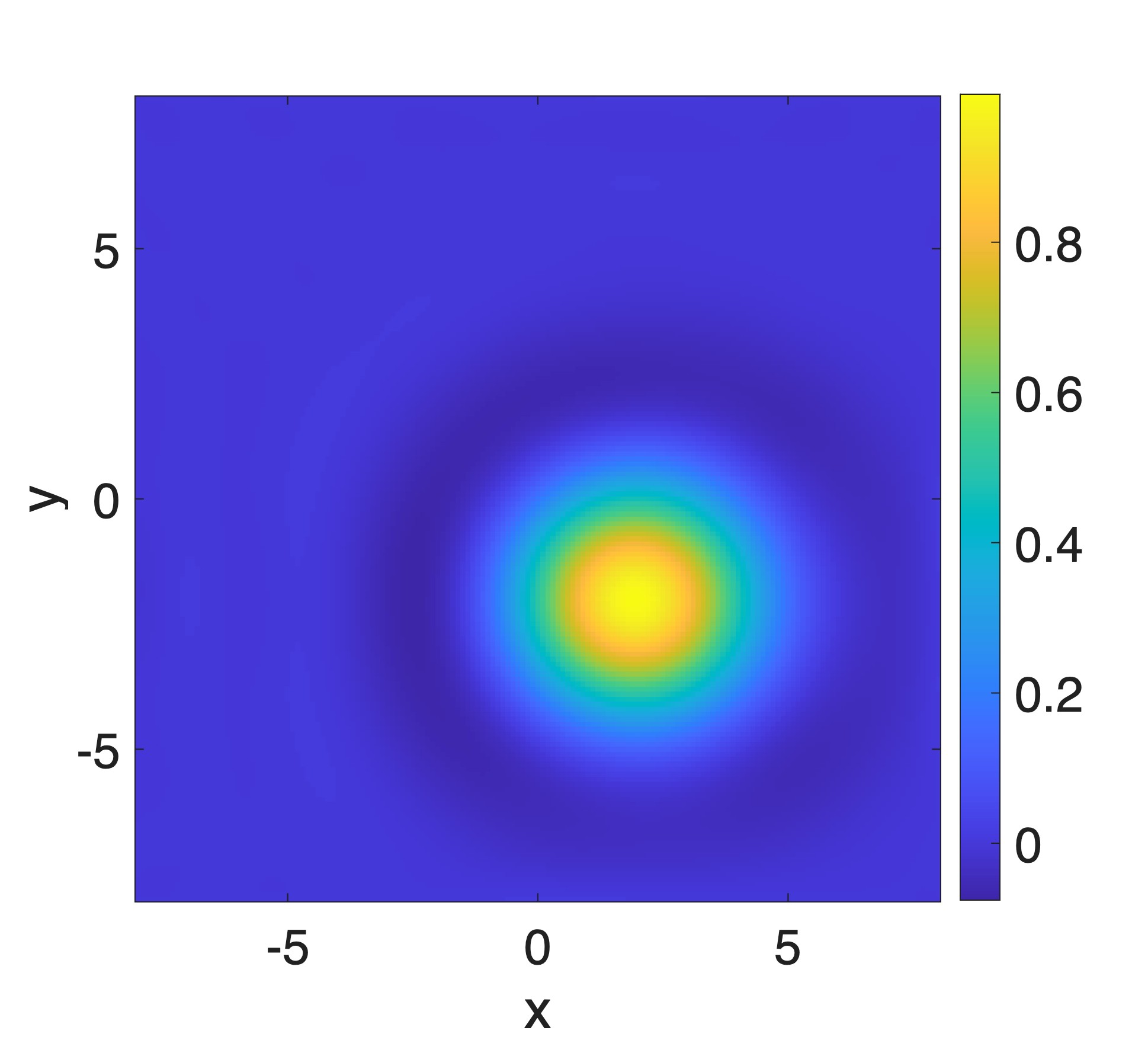}
    }
    \hfill
    \subfloat[Absolute error $|u_0^{\rm rec}-u_0|$, $10\%$ noise]{
        \includegraphics[width=0.31\textwidth]{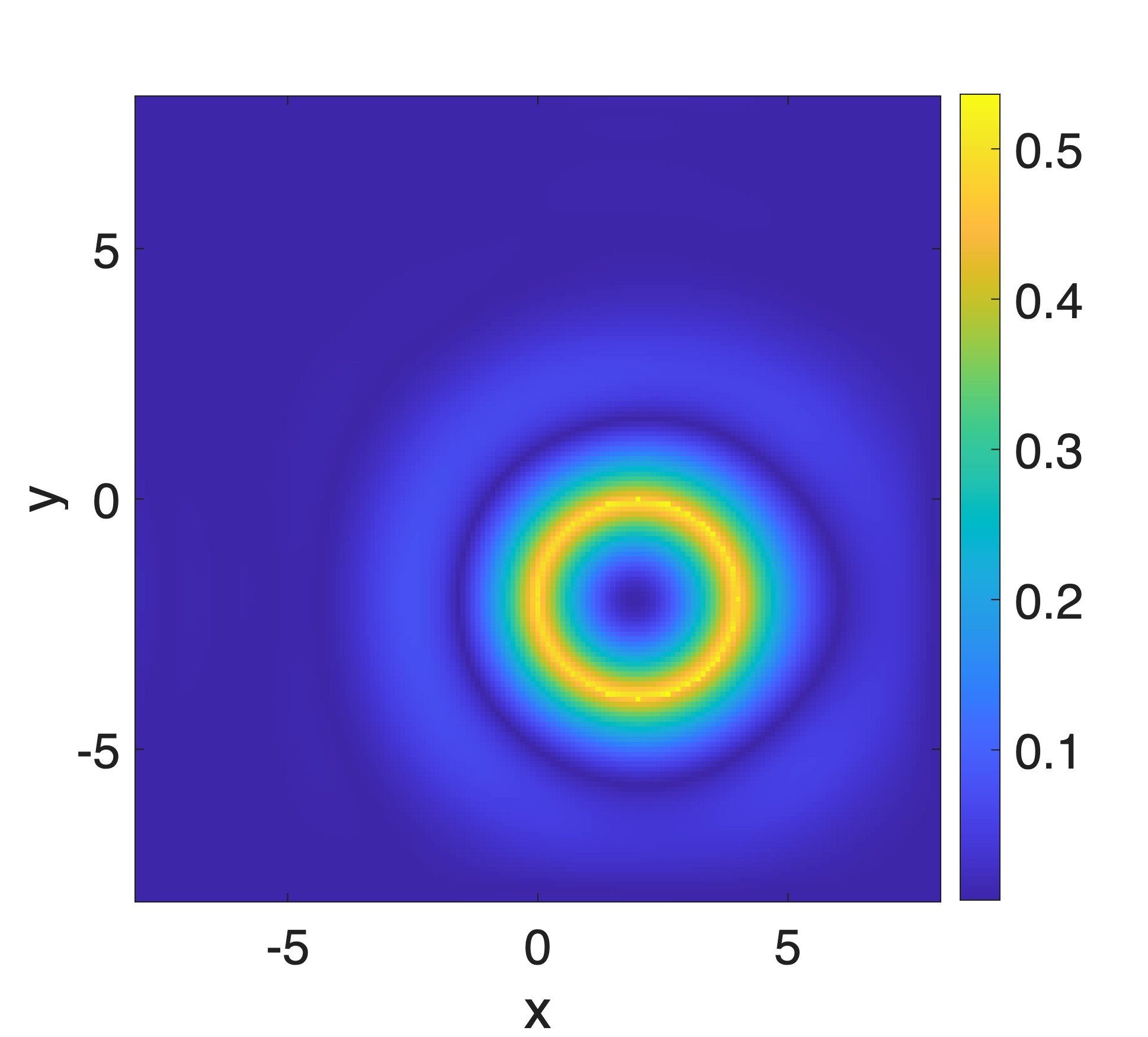}
    }

    \makebox[0.31\textwidth]{}
    \hfill
    \subfloat[Reconstructed initial condition $u_0^{\rm rec}$, $20\%$ noise]{
        \includegraphics[width=0.31\textwidth]{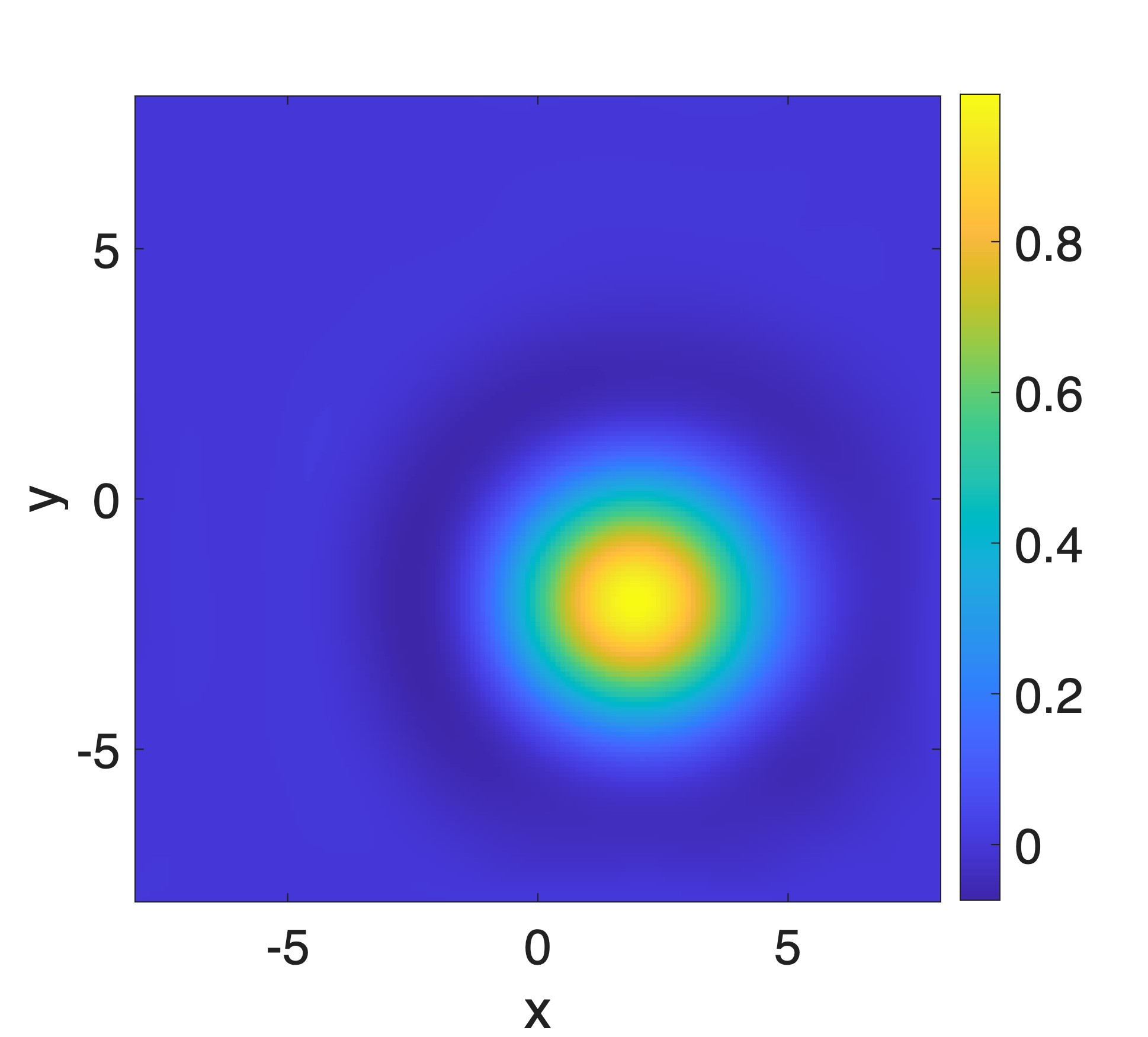}
    }
    \hfill
    \subfloat[Absolute error $|u_0^{\rm rec}-u_0|$, $20\%$ noise]{
        \includegraphics[width=0.31\textwidth]{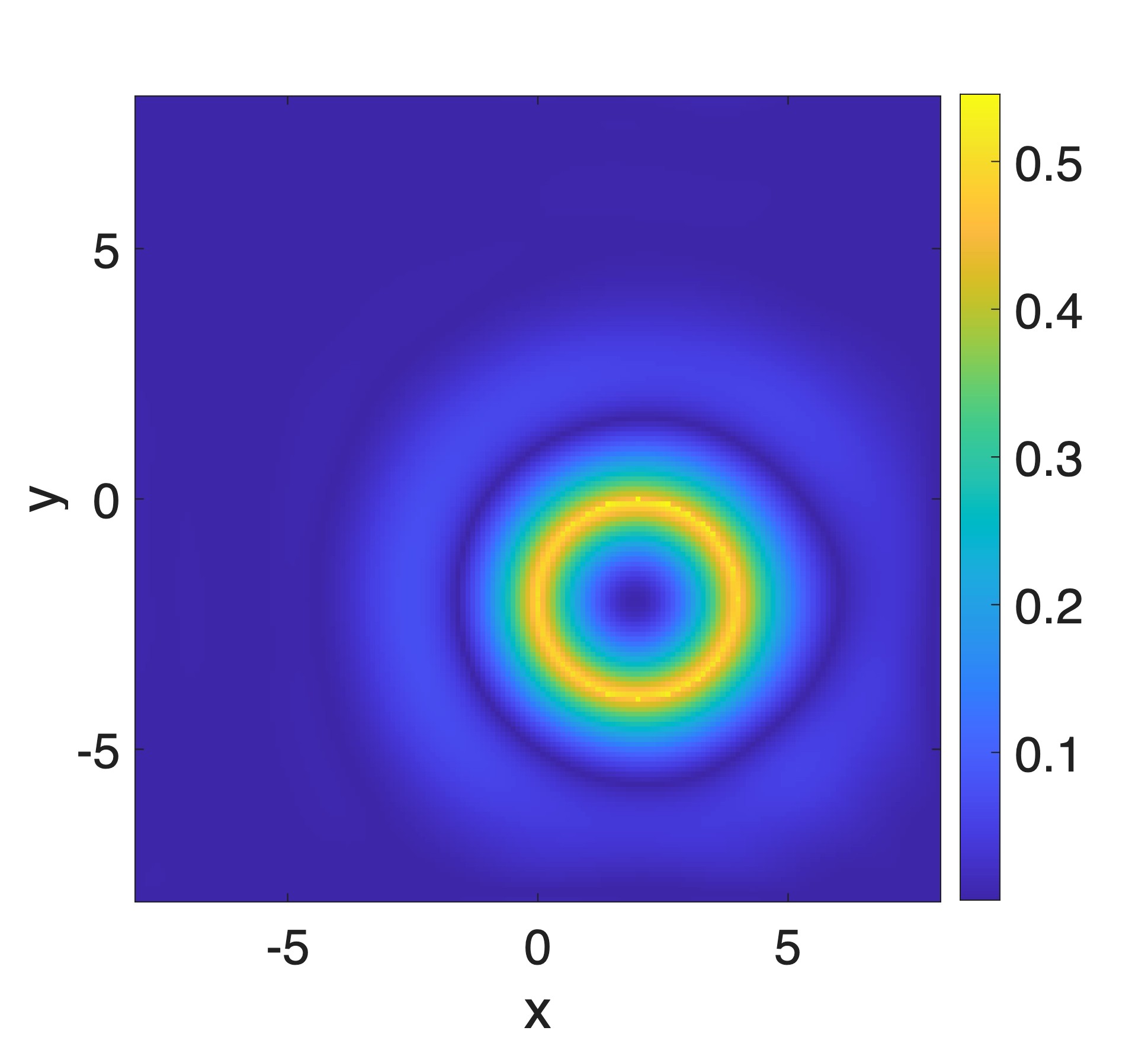}
    }

   \caption{Test 1. Reconstruction results with $10\%$ and $20\%$ noise in the terminal data.} \label{fig:case1-noise10-noise20}
\end{figure}

The reconstruction in Figure \ref{fig:case1-noise10-noise20} captures the location and overall shape of the exact initial profile well for both noise levels. 
As expected, the main error occurs near the inclusion boundary, where the exact profile has a sharp transition. 
To quantify the recovery of the peak amplitude, we define the relative maximum error by
\[
    E_{\max}=\frac{|\max u_0-\max u_0^{\rm rec}|}{\max u_0}.
\]
For $10\%$ noise, the maximum value of the reconstructed initial condition is $0.997039$, with $E_{\max}=0.30\%$. 
For $20\%$ noise, the maximum value of the reconstructed initial condition is $0.995763$, with $E_{\max}=0.42\%$. 
Thus, even when the noise level is increased from $10\%$ to $20\%$, the proposed method still accurately recovers the peak amplitude, while the largest pointwise errors remain concentrated near the jump interface.

\textbf{Test 2.}
In this test, the exact initial condition consists of two elliptical inclusions with opposite signs. 
The first ellipse is centered at $(-3.5,-3.5)$ with semi-axes $4$ and $2$, and has value $-2$. 
The second ellipse is centered at $(3,3)$ with semi-axes $2$ and $4$, and has value $2$. 
More precisely,
\[
    u_0(x,y)
    =
    \begin{cases}
    -2,
    &
    \left(\dfrac{x+3.5}{4}\right)^2
    +
    \left(\dfrac{y+3.5}{2}\right)^2
    \leq 1,
    \\[0.35cm]
    2,
    &
    \left(\dfrac{x-3}{2}\right)^2
    +
    \left(\dfrac{y-3}{4}\right)^2
    \leq 1,
    \\[0.35cm]
    0,
    & \text{otherwise}.
    \end{cases}
\]

The coefficient of the Laplacian term is
\[
    a(x,y,t)
    =
    a_0(x,y)
    \left[
    1+0.10\sin\left(2\pi t+\frac{\pi}{6}\right)
    \right],
\]
where
\[
    a_0(x,y)
    =
    \begin{cases}
       0.9
    +
    0.15\sin\left(\frac{\pi x}{15}\right)
    \cos\left(\frac{\pi y}{15}\right),
    &  x^2+(y-3)^2\geq 0.7^2,
    \\
        0, &\text{otherwise},
    \end{cases}
\]
This test includes a region where the coefficient $a$ is not strictly positive.

The convection field is
\[
    b_x(x,y,t)
    =
    0.35\cos\left(\frac{\pi y}{15}\right)
    \left[
    1+0.18\sin\left(2\pi t+\frac{\pi}{3}\right)
    \right],
\]
and
\[
    b_y(x,y,t)
    =
    -0.35\sin\left(\frac{\pi x}{15}\right)
    \left[
    1+0.18\sin\left(2\pi t+\frac{\pi}{2}\right)
    \right].
\]
The memory kernel is
\[
    \alpha(x,y,s)
    =
    \left[
    0.12
    +
    0.04\cos\left(\frac{\pi x}{15}\right)
    \cos\left(\frac{\pi y}{15}\right)
    \right]
    e^{-s}
    \left[
    1+0.25\cos(2\pi s)
    \right],
    \qquad s\geq 0.
\]

Figure \ref{fig:case2-noise10-noise20} shows the reconstruction results for Test 2 with $10\%$ and $20\%$ noise in the terminal data. 
The exact initial condition contains two discontinuous elliptical inclusions with opposite signs. 
For both noise levels, the reconstruction captures the positive and negative components, including their locations, orientations, and approximate amplitudes. 
As expected, the largest pointwise errors occur near the boundaries of the two ellipses, where the exact initial condition has jump discontinuities. 
Away from these interfaces, the reconstruction error remains relatively small.

\begin{figure}[H]
    \subfloat[Exact initial condition $u_0$]{
        \includegraphics[width=0.31\textwidth]{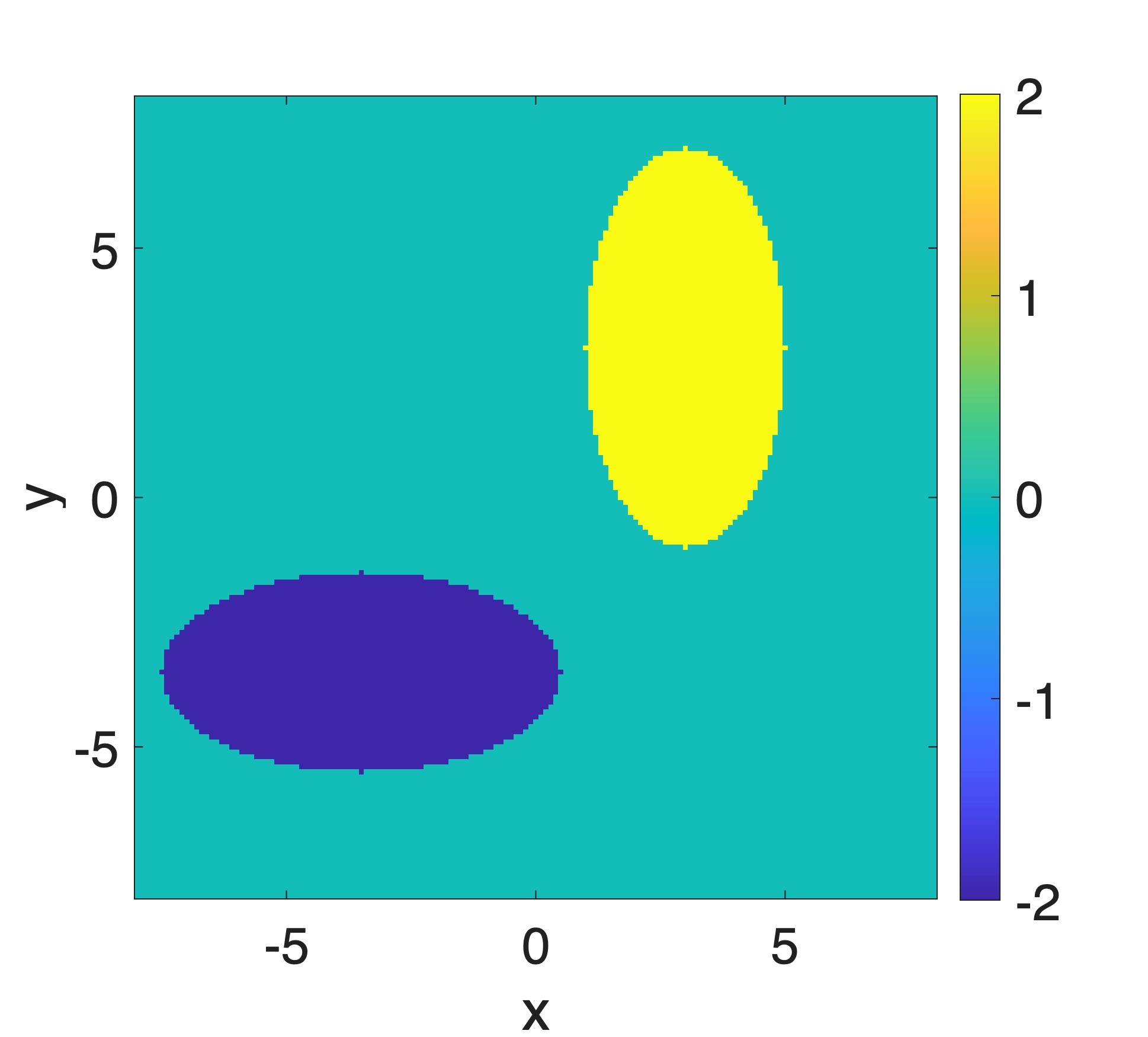}
    }
    \hfill
    \subfloat[Reconstructed initial condition $u_0^{\rm rec}$, $10\%$ noise]{
        \includegraphics[width=0.31\textwidth]{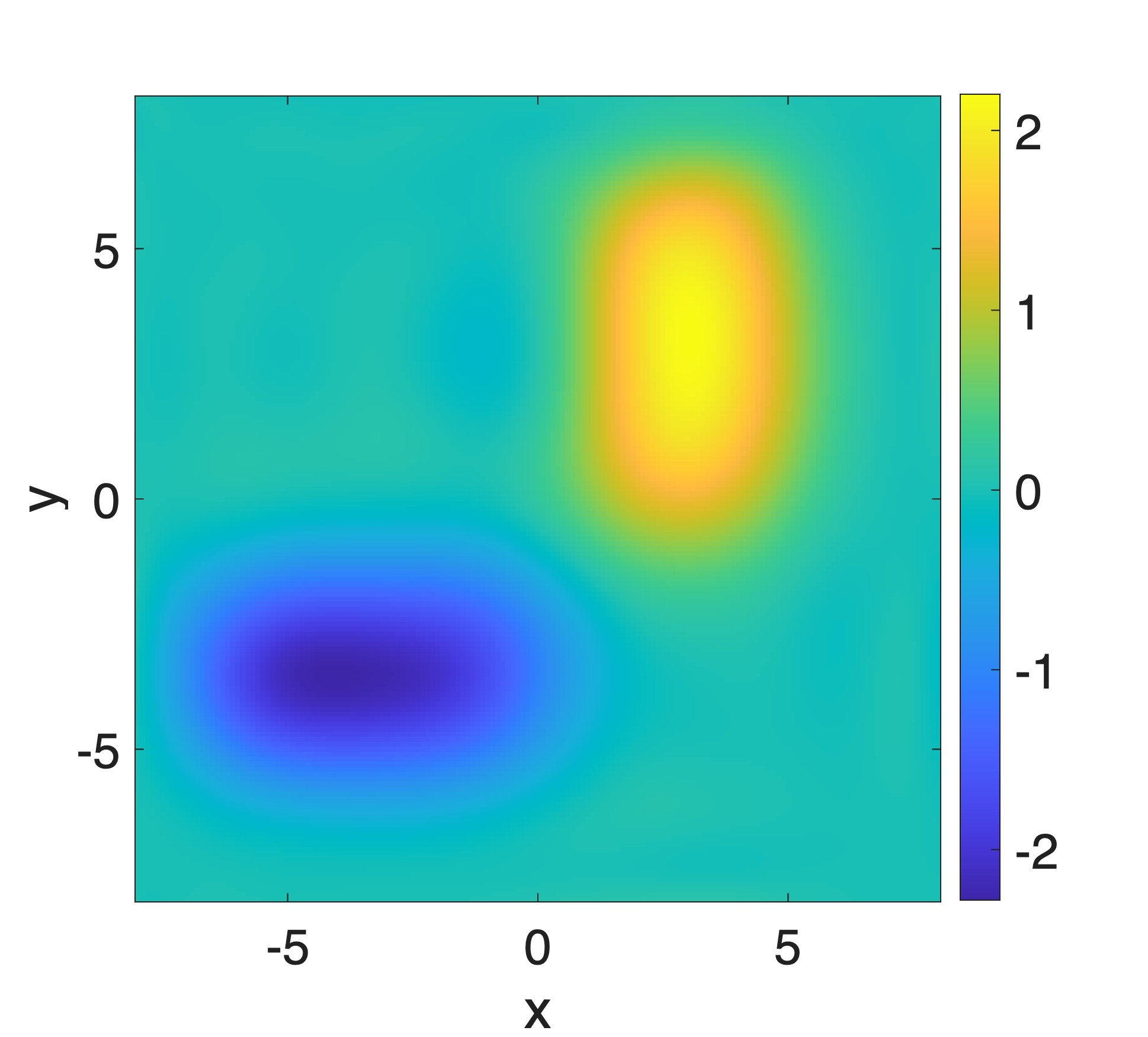}
    }
    \hfill
    \subfloat[Absolute error $|u_0^{\rm rec}-u_0|$, $10\%$ noise]{
        \includegraphics[width=0.31\textwidth]{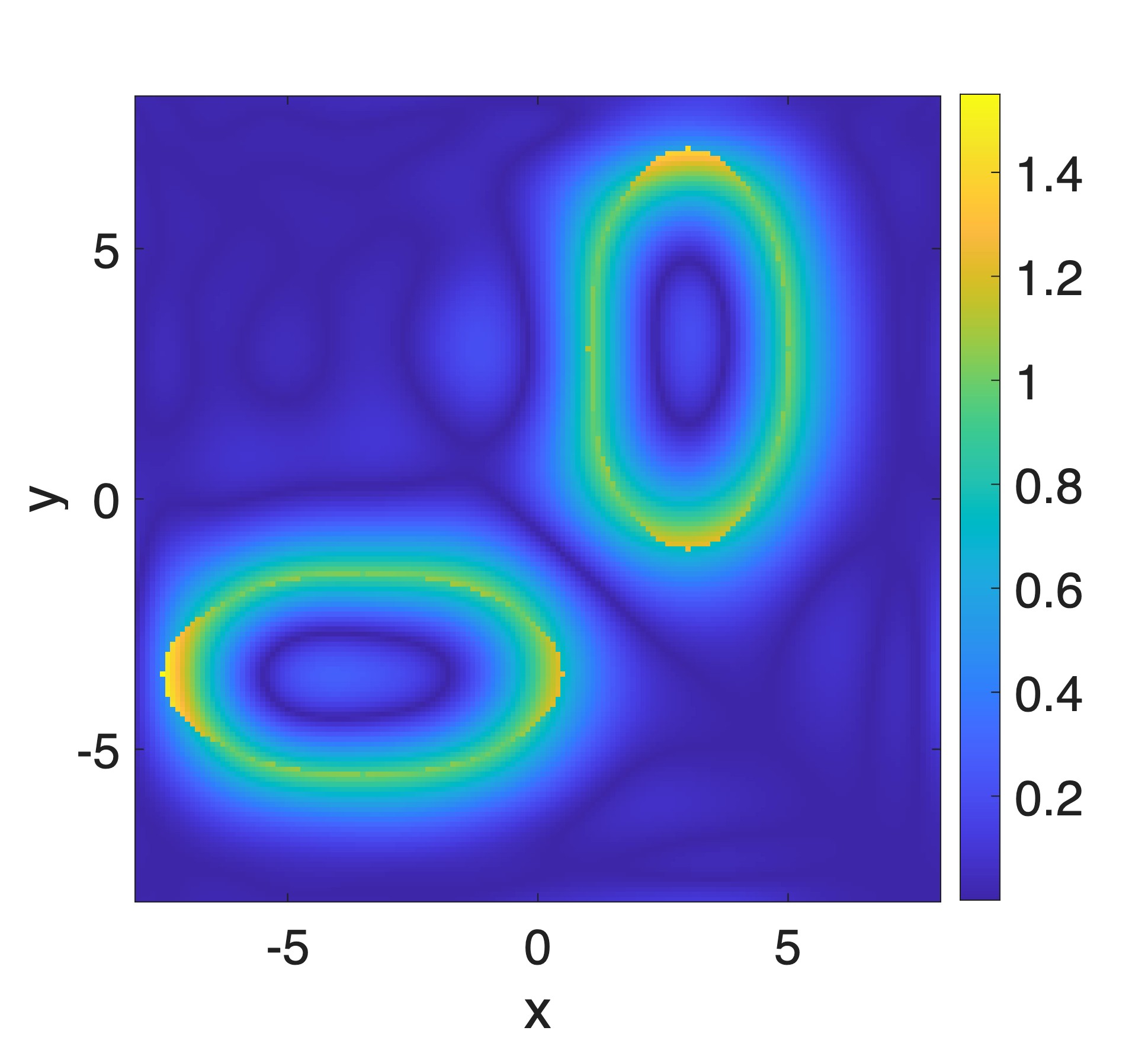}
    }

    \makebox[0.31\textwidth]{}
    \hfill
    \subfloat[Reconstructed initial condition $u_0^{\rm rec}$, $20\%$ noise]{
        \includegraphics[width=0.31\textwidth]{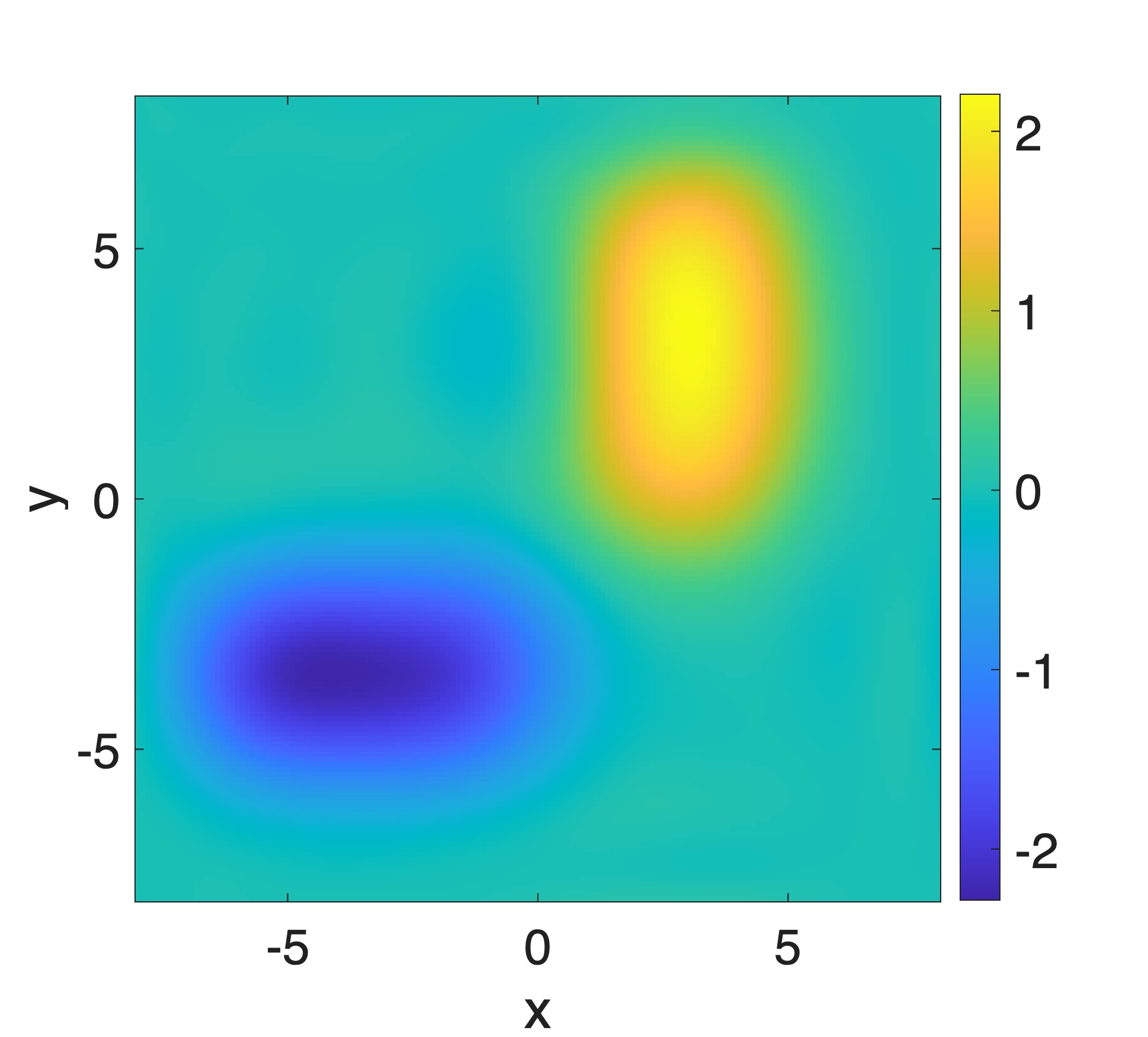}
    }
    \hfill
    \subfloat[Absolute error $|u_0^{\rm rec}-u_0|$, $20\%$ noise]{
        \includegraphics[width=0.31\textwidth]{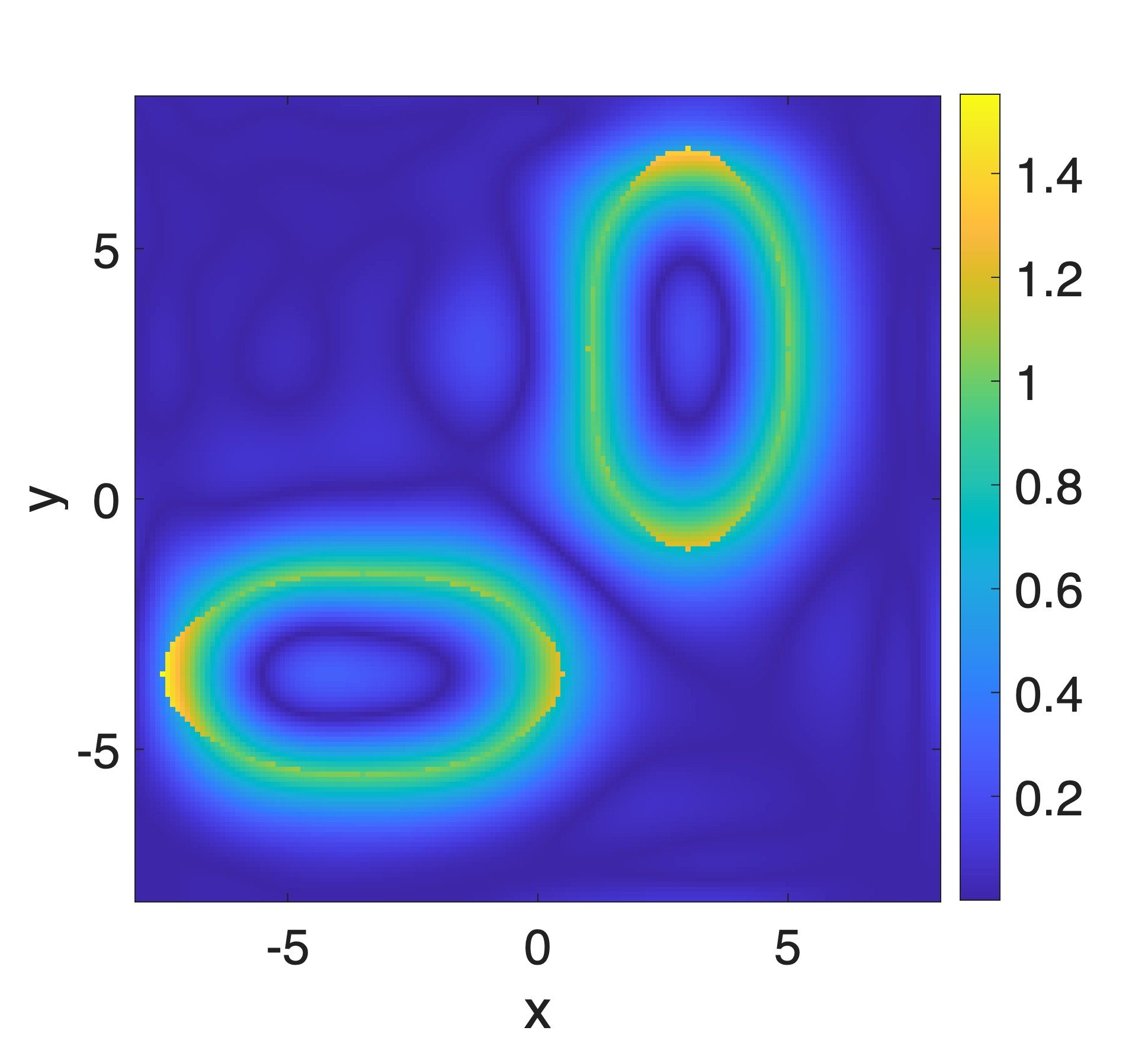}
    }

    \caption{Test 2. Reconstruction results with $10\%$ and $20\%$ noise in the terminal data.} \label{fig:case2-noise10-noise20}
\end{figure}

The reconstruction in Figure \ref{fig:case2-noise10-noise20} captures the main structure of the exact initial condition well for both noise levels. 
Both elliptical inclusions are recovered with the correct locations, orientations, and signs. 
The positive inclusion in the upper-right region and the negative inclusion in the lower-left region are clearly visible in the reconstructed images. 
As expected, the largest errors occur near the boundaries of the two inclusions, where the exact initial condition has jump discontinuities. 

For $10\%$ noise, the maximum value is reconstructed as $2.20155$, compared with the true maximum value $2.00$, giving a relative maximum error of $10.08\%$. 
The minimum value is reconstructed as $-2.2811$, compared with the true minimum value $-2.00$, giving a relative minimum error of $14.05\%$. 
For $20\%$ noise, the maximum value is reconstructed as $2.2068$, giving a relative maximum error of $10.34\%$, while the minimum value is reconstructed as $-2.2845$, giving a relative minimum error of $14.23\%$. 
Overall, the method reconstructs the geometry and contrast of both inclusions reasonably well even when the noise level is increased to $20\%$, although the amplitudes are slightly overestimated.

\textbf{Test 3.}
In this test, the exact initial condition is an annular inclusion with an opening on the right-hand side.
More precisely, 
\[
    u_0(x,y)
    =
    \begin{cases}
    1,
    & 3\leq  \sqrt{x^2+y^2} \leq 6
    \text{ and }
    \left(x\leq 0 \text{ or } |y|>1.35\right),
    \\[0.25cm]
    0,
    & \text{otherwise}.
    \end{cases}
\]
Visually, the true initial profile resembles the letter ``C''. 
It consists of an annular region with a horizontal opening on the right-hand side.

The coefficient of the Laplacian term is
\[
    a(x,y,t)
    =
    a_0(x,y)
    \left[
    1+0.12\sin\left(2\pi t+\frac{\pi}{5}\right)
    \right],
\]
where
\[
    a_0(x,y)
    =
    \begin{cases}
    0.75
    +
    0.25\sin\left(\frac{\pi x}{15}\right)
    \cos\left(\frac{\pi y}{15}\right),
    & (x-3)^2+(y-2.5)^2>0.35^2,
    \\
    0,
    & \mbox{otherwise.}
    \end{cases}
\]
Hence, this test includes a small region where the coefficient $a$ vanishes.

The convection field is
\[
    b_x(x,y,t)
    =
    0.4\cos\left(\frac{\pi y}{15}\right)
    \left[
    1+0.20\sin\left(2\pi t+\frac{\pi}{3}\right)
    \right],
\]
and
\[
    b_y(x,y,t)
    =
    0.4\sin\left(\frac{\pi x}{15}\right)
    \left[
    1+0.20\sin\left(2\pi t+\frac{\pi}{2}\right)
    \right].
\]
The spatial profile of the memory kernel is
\[
    \alpha_0(x,y)
    =
    0.10
    +
    0.03\cos\left(\frac{\pi x}{15}\right)
    \cos\left(\frac{\pi y}{15}\right).
\]
Therefore, using the common lag-dependent structure, the memory kernel is
\[
    \alpha(x,y,s)
    =
    \alpha_0(x,y)e^{-s}
    \left[
    1+0.25\cos(2\pi s)
    \right],
    \qquad s\geq 0.
\]

\begin{figure}[H]
    \subfloat[Exact initial condition $u_0$]{
        \includegraphics[width=0.31\textwidth]{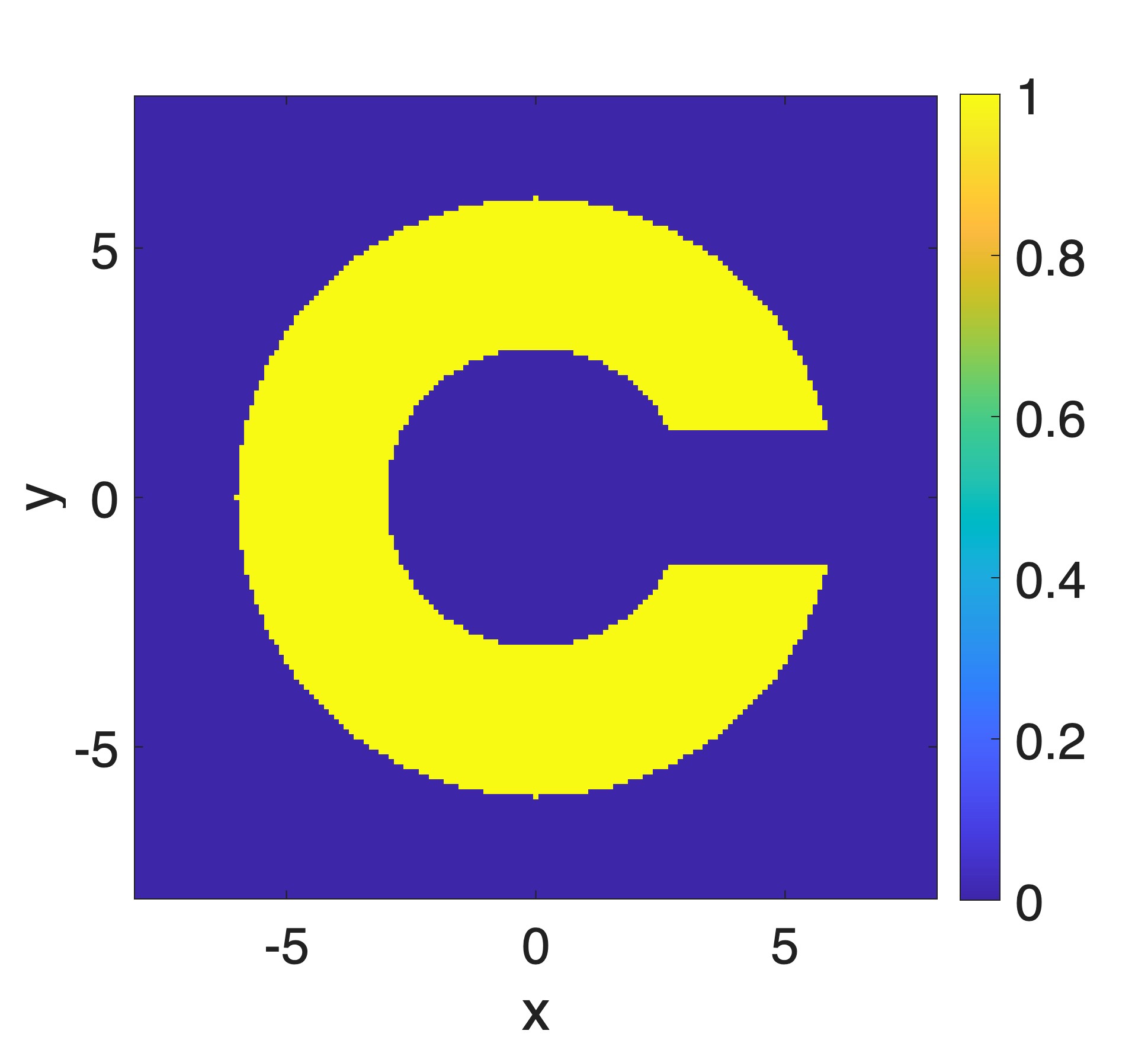}
    }
    \hfill
    \subfloat[Reconstructed initial condition $u_0^{\rm rec}$, $5\%$ noise]{
        \includegraphics[width=0.31\textwidth]{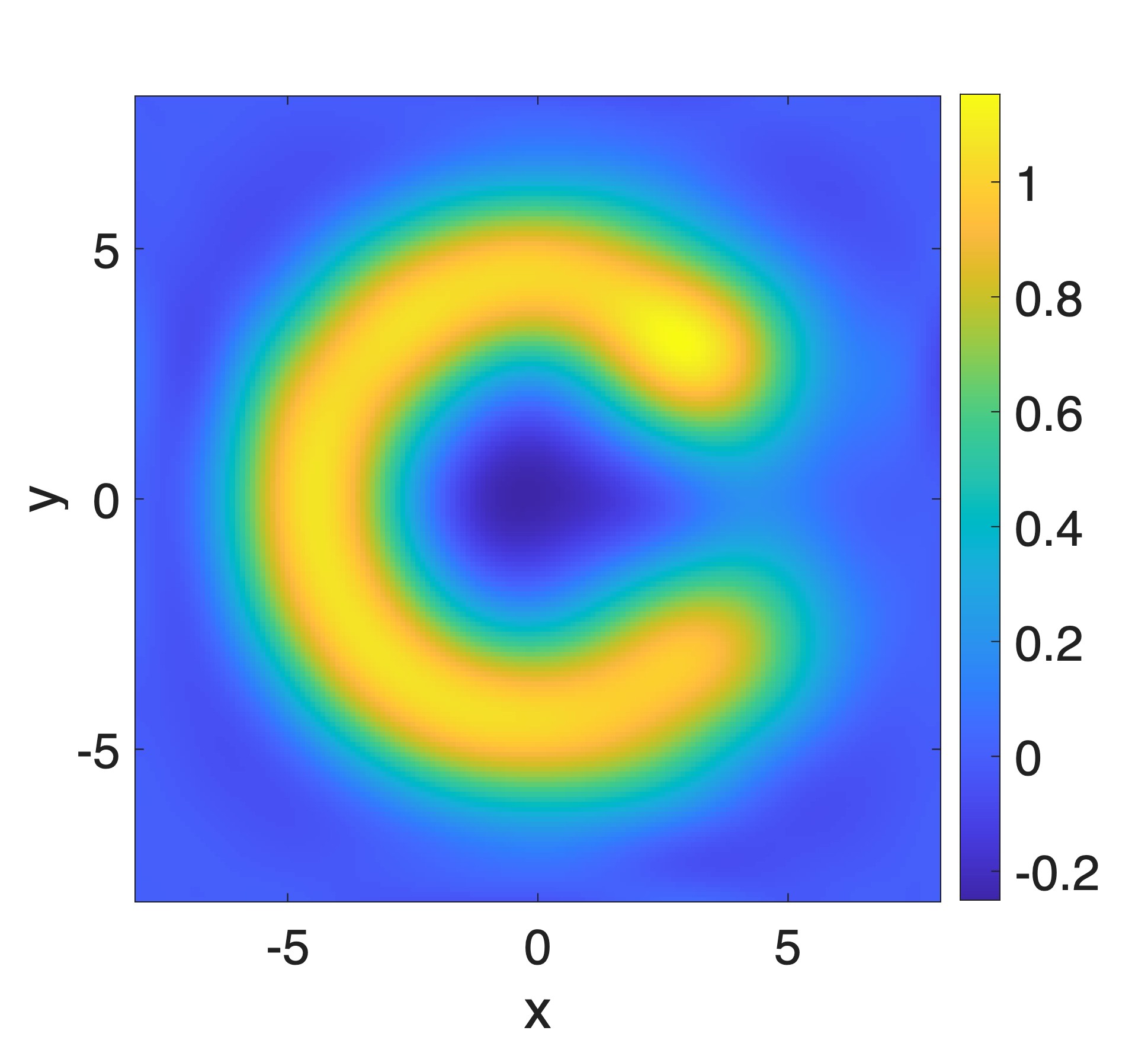}
    }
    \hfill
    \subfloat[Absolute error $|u_0^{\rm rec}-u_0|$, $5\%$ noise]{
        \includegraphics[width=0.31\textwidth]{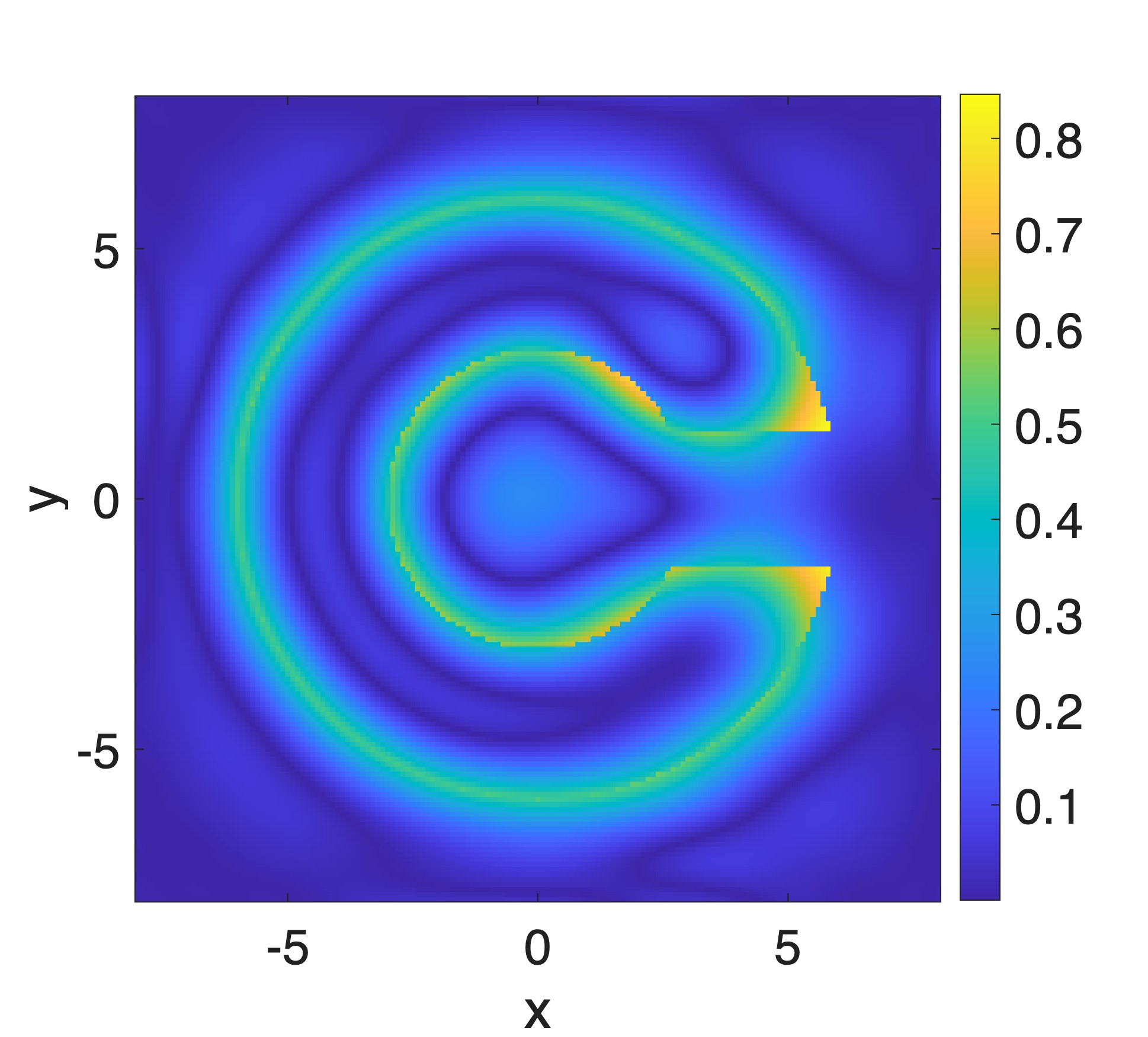}
    }

    \makebox[0.31\textwidth]{}
    \hfill
    \subfloat[Reconstructed initial condition $u_0^{\rm rec}$, $10\%$ noise]{
        \includegraphics[width=0.31\textwidth]{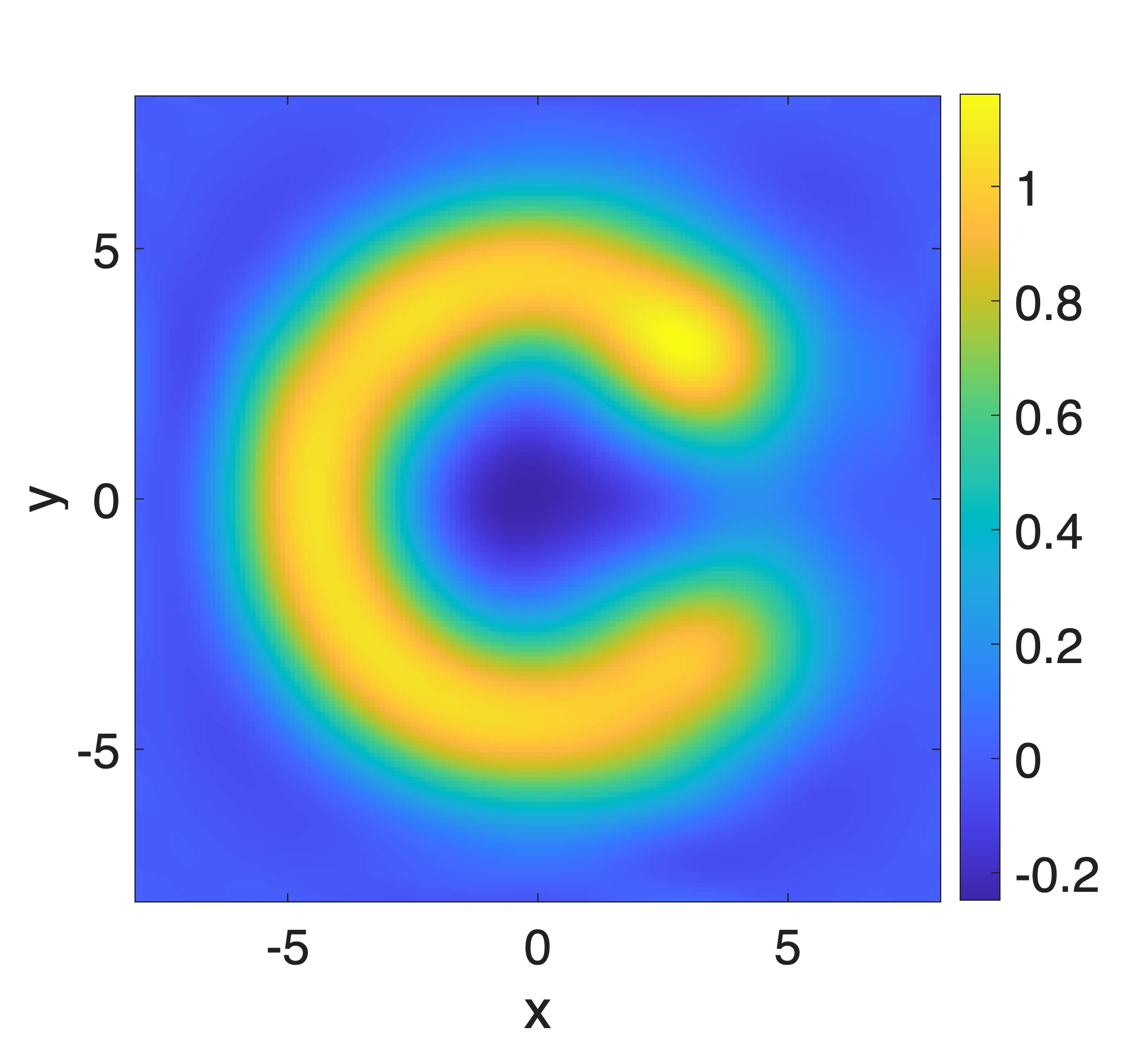}
    }
    \hfill
    \subfloat[Absolute error $|u_0^{\rm rec}-u_0|$, $10\%$ noise]{
        \includegraphics[width=0.31\textwidth]{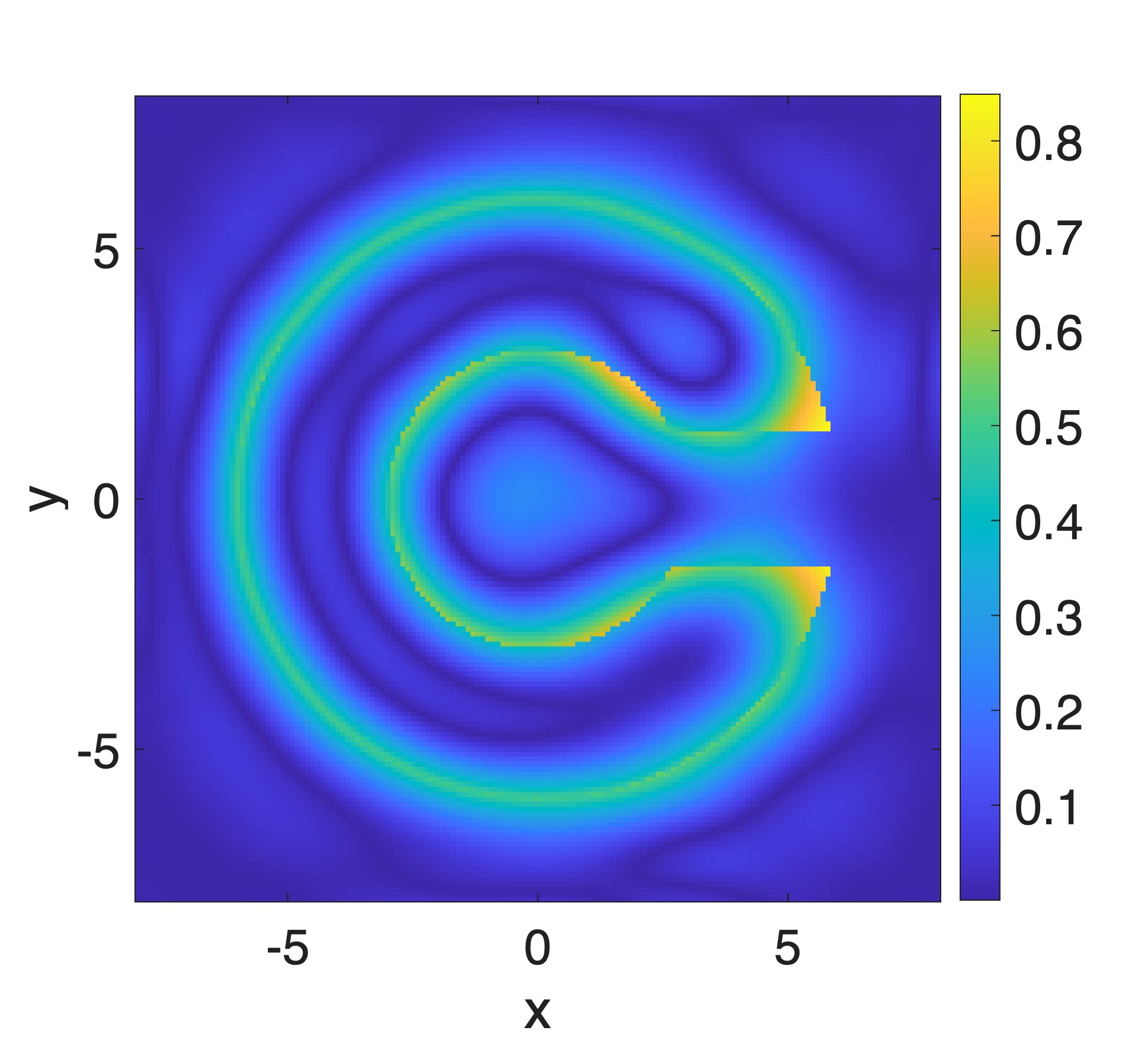}
    }

    \caption{Test 3. Reconstruction results with $5\%$ and $10\%$ noise in the terminal data.} 
    \label{fig:case3-noise5-noise10}
\end{figure}

Figure \ref{fig:case3-noise5-noise10} shows the reconstruction results for Test 3 with $5\%$ and $10\%$ noise in the terminal data. 
The exact initial condition is a discontinuous ring-shaped profile with an opening on the right-hand side; visually, it resembles the letter ``C''. 
For both noise levels, the reconstruction captures the overall C-shaped geometry, including the inner void and the right-hand opening. 
As in the previous tests, the main errors occur near the jump boundaries of the initial profile, especially along the inner and outer interfaces and near the two ends of the opening. 
Away from these interfaces, the reconstruction error remains relatively small.

Visually, the reconstruction in Figure \ref{fig:case3-noise5-noise10} is quite satisfactory for both noise levels. 
The method successfully recovers the main C-shaped structure, including the inner void and the opening on the right-hand side. 
The reconstructed images preserve the overall geometry of the true initial condition, although the sharp interfaces are smoothed, as expected for a backward reconstruction from noisy terminal data. 
The largest pointwise errors are concentrated near the jump boundaries, especially along the inner and outer edges of the ring and near the two ends of the opening. 
For $5\%$ noise, the maximum value of the reconstructed initial condition is $1.152836$, giving $E_{\max}=15.28\%$. 
For $10\%$ noise, the maximum value of the reconstructed initial condition is $1.16$, giving $E_{\max}=16.13\%$. 
Overall, despite the discontinuous C-shaped profile, the reconstruction captures the target's essential shape and location well for both noise levels.

\textbf{Test 4.}
In this test, the exact initial condition is a square annulus. 
It consists of a large square centered at the origin with a smaller square hole removed from the center. 
More precisely,
\[
    u_0(x,y)
    =
    \begin{cases}
    1,
    & |x|\leq 6,\ |y|\leq 6,\ \text{and } \bigl(|x|>3 \text{ or } |y|>3\bigr),
    \\[0.25cm]
    0,
    & \text{otherwise}.
    \end{cases}
\]
Thus, $u_0$ is discontinuous across the boundaries of the outer and inner squares.

The coefficient of the Laplacian term is
\[
    a(x,y,t)
    =
    a_0(x,y)
    \left[
    1+0.10\sin\left(2\pi t+\frac{\pi}{7}\right)
    \right],
\]
where
\[
    a_0(x,y)
    =
    \begin{cases}
    0.8
    +
    0.1\sin\left(\frac{\pi x}{15}\right)
    \cos\left(\frac{\pi y}{15}\right),
    &
    |x-1.2|>2.2
    \text{ or }
    |y+0.8|>1.6,
    \\[0.25cm]
    0,
    &
    |x-1.2|\leq 2.2
    \text{ and }
    |y+0.8|\leq 1.6.
    \end{cases}
\]
Hence this test includes a rectangular region where the coefficient $a$ vanishes.

The convection field is
\[
    b_x(x,y,t)
    =
    0.3\cos\left(\frac{\pi y}{15}\right)
    \left[
    1+0.18\sin\left(2\pi t+\frac{\pi}{4}\right)
    \right],
\]
and
\[
    b_y(x,y,t)
    =
    0.3\sin\left(\frac{\pi x}{15}\right)
    \left[
    1+0.18\sin\left(2\pi t+\frac{\pi}{2}\right)
    \right].
\]
The spatial profile of the memory kernel is
\[
    \alpha_0(x,y)
    =
    0.12
    +
    0.03\cos\left(\frac{\pi x}{15}\right)
    \cos\left(\frac{\pi y}{15}\right).
\]
Therefore,
\[
    \alpha(x,y,s)
    =
    \alpha_0(x,y)e^{-s}
    \left[
    1+0.25\cos(2\pi s)
    \right],
    \qquad s\geq 0.
\]

\begin{figure}[H]
    \subfloat[Exact initial condition $u_0$]{
        \includegraphics[width=0.31\textwidth]{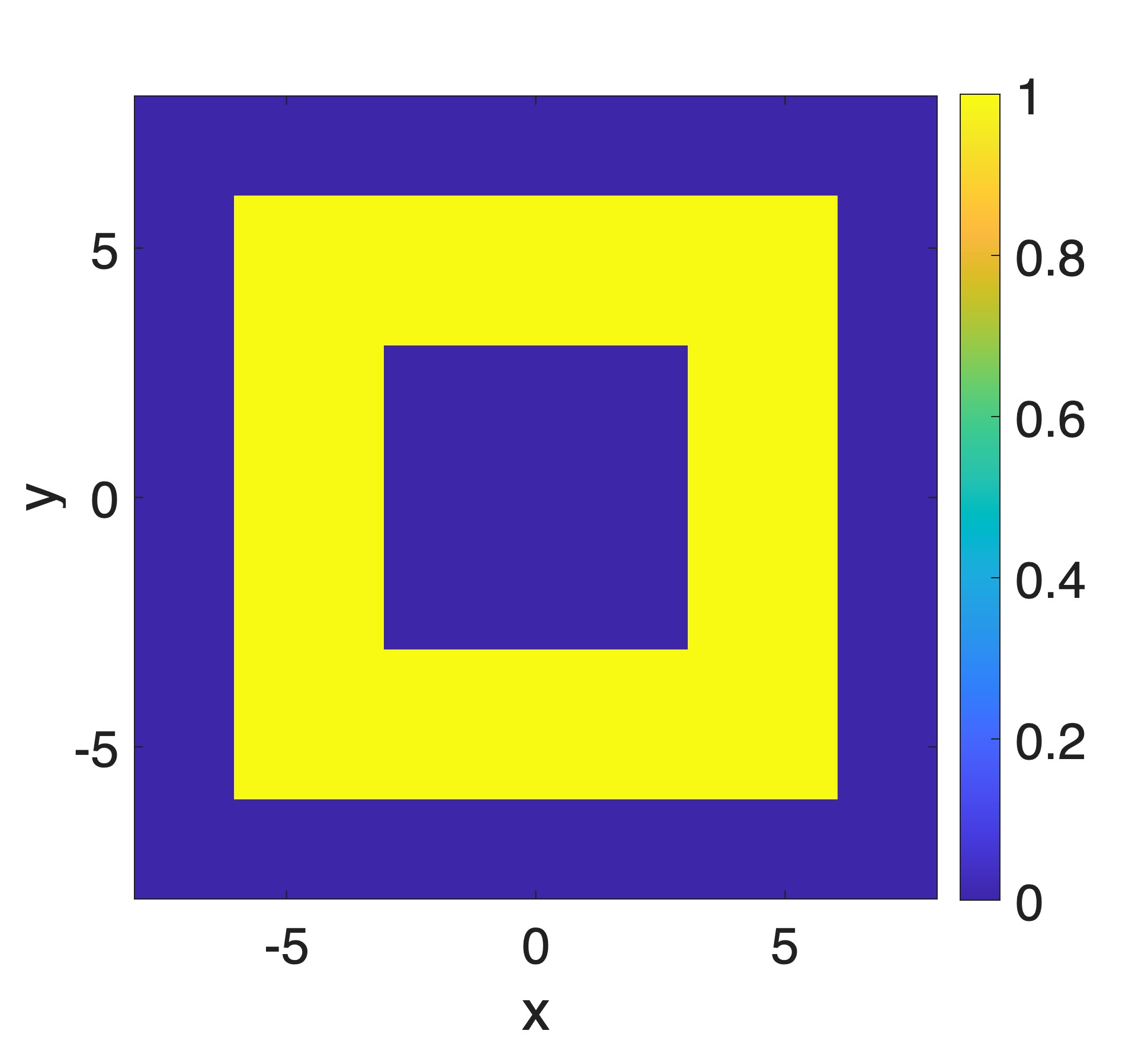}
    }
    \hfill
    \subfloat[Reconstructed initial condition $u_0^{\rm rec}$, $10\%$ noise]{
        \includegraphics[width=0.31\textwidth]{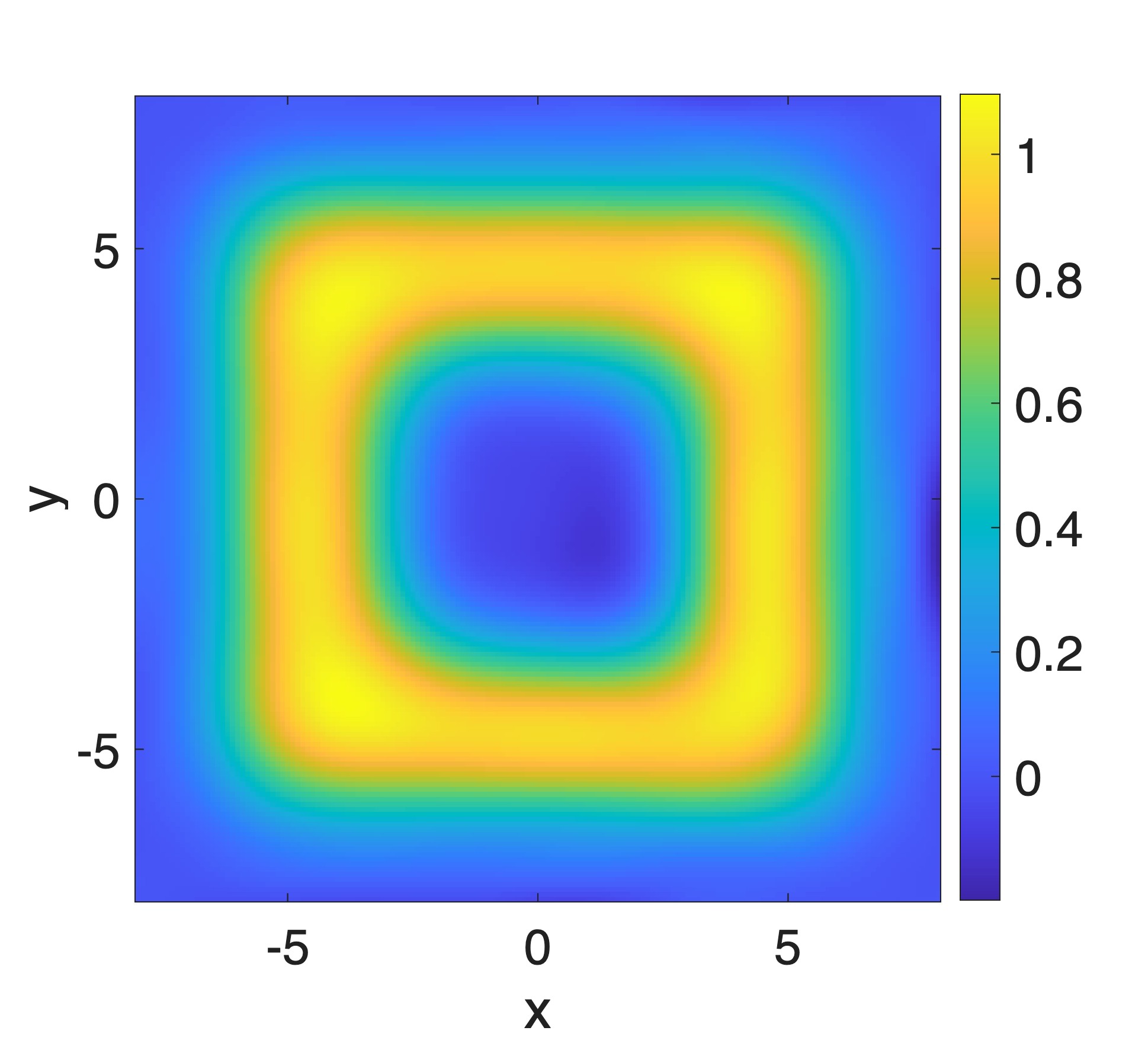}
    }
    \hfill
    \subfloat[Absolute error $|u_0^{\rm rec}-u_0|$, $10\%$ noise]{
        \includegraphics[width=0.31\textwidth]{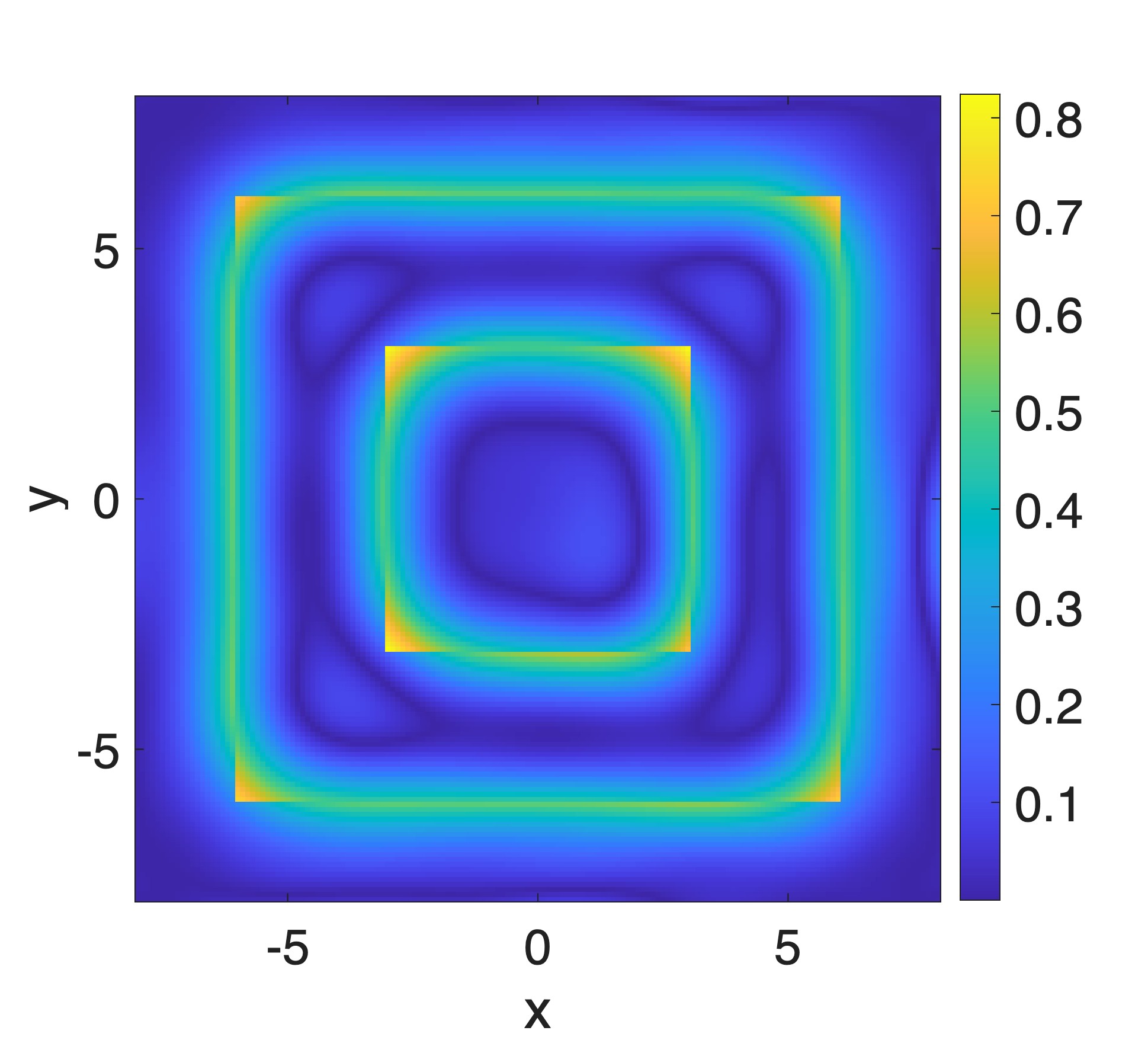}
    }

    \makebox[0.31\textwidth]{}
    \hfill
    \subfloat[Reconstructed initial condition $u_0^{\rm rec}$, $20\%$ noise]{
        \includegraphics[width=0.31\textwidth]{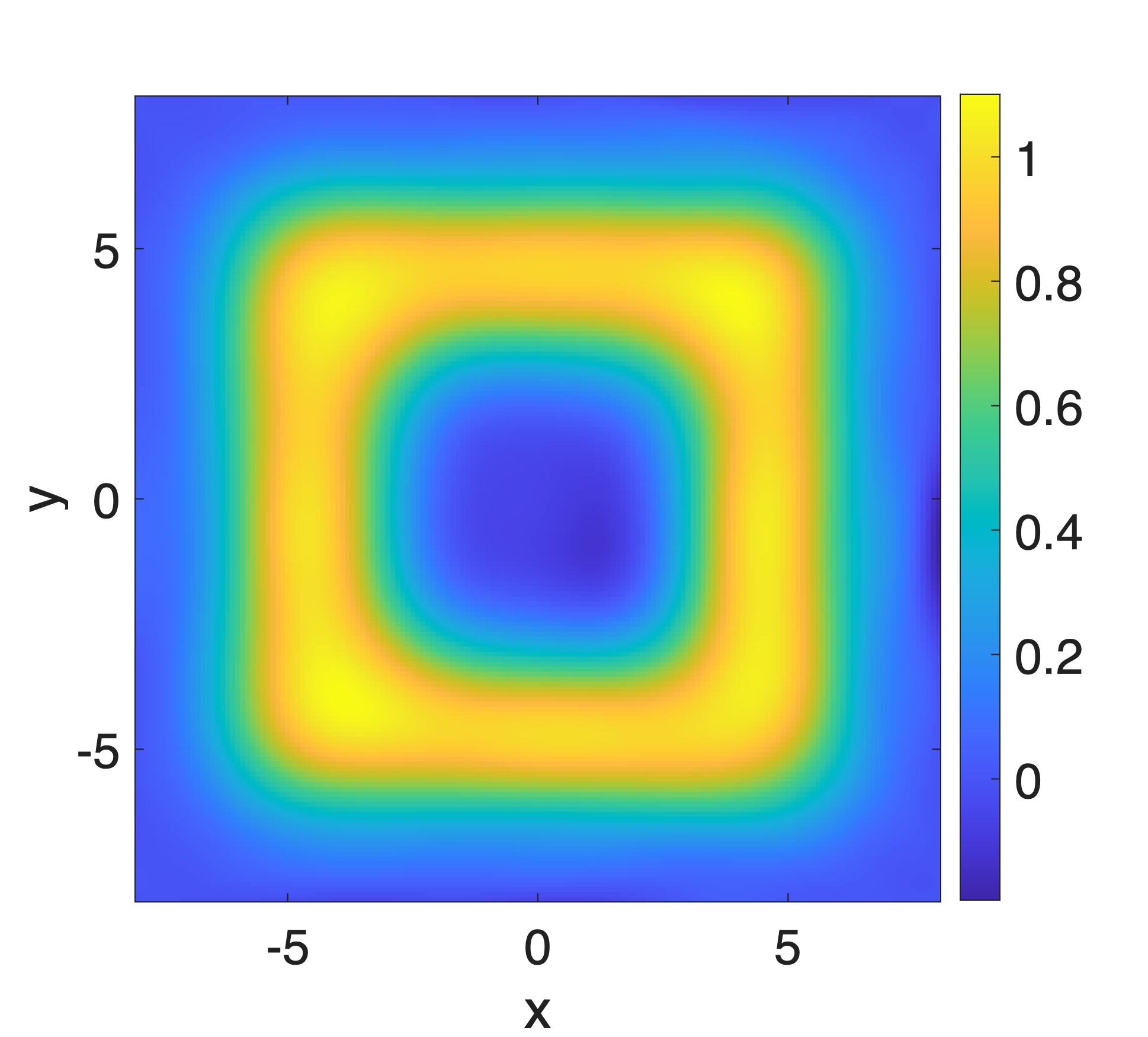}
    }
    \hfill
    \subfloat[Absolute error $|u_0^{\rm rec}-u_0|$, $20\%$ noise]{
        \includegraphics[width=0.31\textwidth]{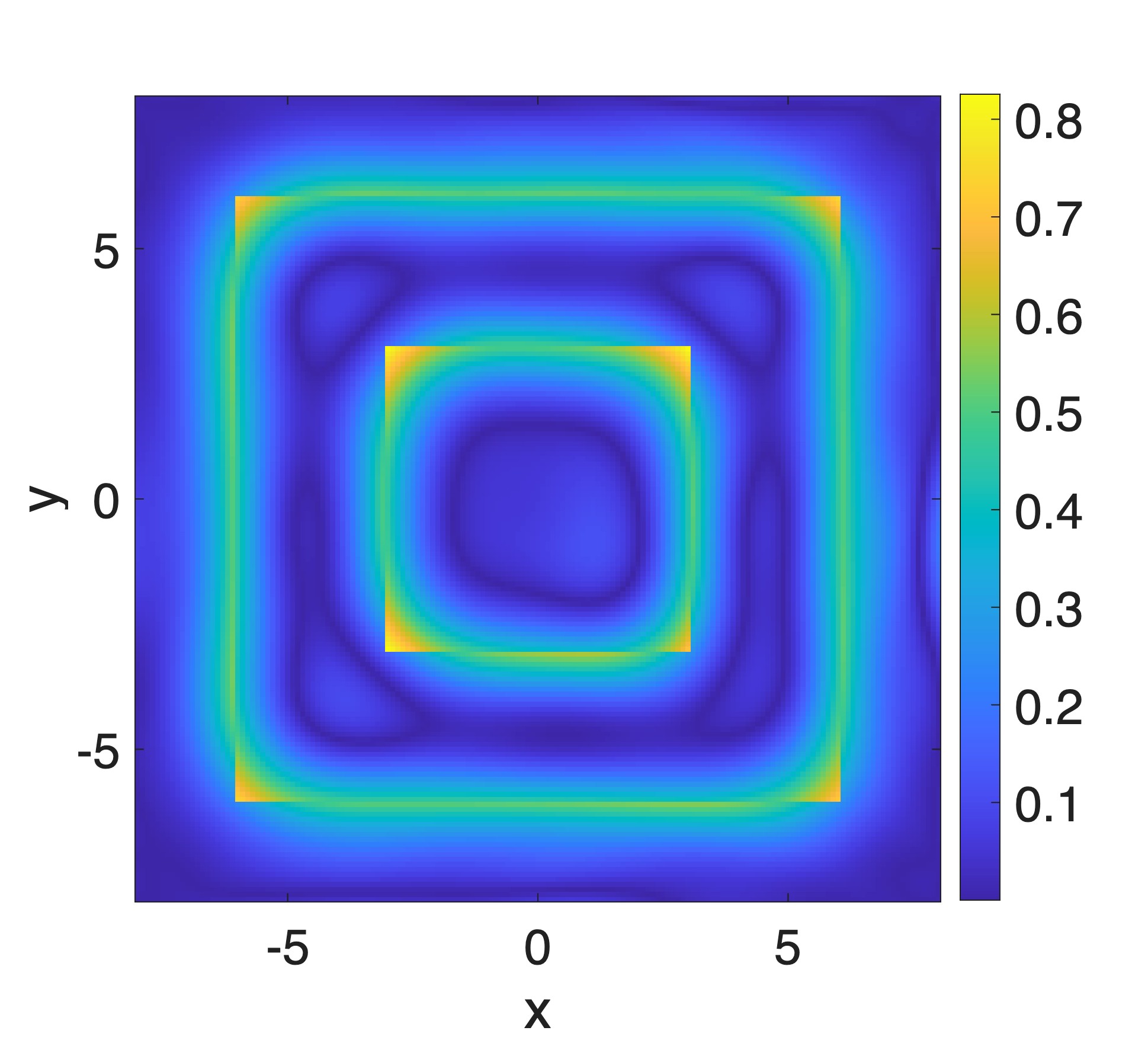}
    }

    \caption{Test 4. Reconstruction results with $10\%$ and $20\%$ noise in the terminal data.} \label{fig:case4-noise10-noise20}
\end{figure}

Figure \ref{fig:case4-noise10-noise20} shows the reconstruction results for Test 4 with $10\%$ and $20\%$ noise in the terminal data. 
The exact initial condition is a discontinuous square annulus. 
For both noise levels, the reconstruction captures the main square-ring structure well, including the outer square boundary and the inner square void. 
The reconstructed images are smoother than the exact profile, which is expected because the inverse problem is ill-posed and the data are noisy. 
The largest pointwise errors occur near jump interfaces, particularly along the inner and outer square boundaries and at the corners. 
Away from these discontinuity curves, the reconstruction error remains relatively small.

Visually, the reconstruction in Figure \ref{fig:case4-noise10-noise20} is satisfactory for both noise levels. 
The method recovers the square-annulus structure, including the central void and the outer square boundary. 
As expected, the reconstructed profiles are smoother than the exact discontinuous profile, and the main errors are concentrated near the jump interfaces, especially around the inner and outer square boundaries. 
For $10\%$ noise, the maximum value of the reconstructed initial condition is $1.10$, giving $E_{\max}=9.65\%$. 
For $20\%$ noise, the maximum value of the reconstructed initial condition is $1.100640$, giving $E_{\max}=10.06\%$. 
Overall, the reconstruction captures the main geometry of the target well for both noise levels, despite the sharp corners of the discontinuous square-annulus profile.

\begin{Remark}
The comparable reconstruction errors for different noise levels can be explained by the filtering effect of the spatial truncation. The Legendre reduction keeps only finitely many low-frequency spatial modes and removes the high-frequency components of the terminal data, where the backward instability is most severe. As a result, the reduced inverse problem is less sensitive to high-frequency noise. In the tests above, increasing the noise level does not significantly change the recovered large-scale geometry, because the dominant high-frequency perturbations are filtered out by the truncation and further damped by the Tikhonov penalty. The remaining errors are mainly caused by smoothing at jump interfaces and by finite-dimensional approximation.
\end{Remark}

\section{Concluding remarks}
\label{sec:conclusion}

We studied an inverse initial data problem for a convection--diffusion equation with memory. 
The problem is ill-posed because final-time data contain only smoothed information about the initial state. 
We proved uniqueness in a spatially independent coefficient setting by using the Fourier transform and an analyticity argument for the associated scalar Volterra equation.

For the variable-coefficient case, we developed a Legendre-based spatial-dimensional reduction method combined with Tikhonov regularization. 
The terminal condition is imposed exactly, while the reduced integro-differential system is enforced in the least-squares sense with an $H^2$ penalty. 
For a fixed truncation order, we proved the convergence of the regularized minimizers to the finite-dimensional minimum-norm solution.

The numerical examples show that the method can recover several discontinuous initial profiles from noisy terminal data, including circular, elliptical, C-shaped, and square-annulus targets. 
The main geometric features are recovered well, and the largest errors occur near jump discontinuities, as expected. 
The tests also indicate that the method can be applied when the coefficient of the Laplacian term is not uniformly positive.

Several questions remain beyond the scope of the present paper. 
In particular, the theoretical analysis in this work is limited to the spatially independent-coefficient case, whereas the fully variable-coefficient case is investigated numerically. 
Extending the uniqueness and stability theory to the fully variable-coefficient setting is an important direction for future work.

\end{document}